\documentclass[letter,12pt]{article}

\usepackage{amssymb}%*
\usepackage{amsthm}
\usepackage{amsfonts}%*
\usepackage[latin2]{inputenc}%*
\usepackage{amsmath}%*
\usepackage{anysize}%*
\usepackage{hyperref}
\usepackage{bbm}
\marginsize{2.6cm}{2.6cm}{1cm}{2cm}

\newtheorem*{theorem*}{Theorem}
\newtheorem{main_theorem}{Theorem}
\newtheorem{theorem}{Theorem}[section]

\newtheorem{main_corollary}[main_theorem]{Corollary}

\newtheorem{lemma}[theorem]{Lemma}
\newtheorem{corollary}[theorem]{Corollary}
\newtheorem{remark}[theorem]{Remark}

\newtheorem{proposition}[theorem]{Proposition}

\def\del{\partial}
\def\dbar{\bar\partial}
\def\ddbar{\del\dbar}

\def\del{\partial}

\def\o{\omega}

\begin{document}
\title{The Mabuchi Completion of the Space of K\"ahler Potentials}
\author{Tam\'as Darvas\thanks{Research supported by NSF grant DMS1162070. \newline  2010 Mathematics subject classification 53C55, 32W20, 32U05. }}
\date{\vspace{0.0in} \emph{\small{To Andrea.}}\vspace{-0.3in}}
\maketitle
\begin{abstract}
Suppose $(X,\o)$ is a compact K\"ahler manifold. Following Mabuchi, the space of smooth K\"ahler potentials $\mathcal H$ can be endowed with a Riemannian structure, which induces an infinite dimensional path length metric space $(\mathcal H,d)$. We prove that the metric completion of $(\mathcal H,d)$ can be identified with $(\mathcal E^2(X,\o),\tilde d)$, and this latter space is a complete non-positively curved geodesic metric space. In obtaining this result, we will rely on envelope techniques which allow for a treatment in a very general context. Profiting from this, we will characterize the pairs of potentials in $\text{PSH}(X,\o)$ that can be connected by weak geodesics and we will also give a characterization of $\mathcal E(X,\o)$ in this context.
\end{abstract}

\section{Introduction}

Given $(X^n,\o)$, a connected compact K\"ahler manifold, the space of smooth K\"ahler potentials is the set
$$\mathcal H = \{ u \in C^\infty(X) | \ \o + i\partial \bar \partial u > 0\}.$$
Using the $\ddbar$--lemma of Hodge theory, this set can be identified (up to a constant) with the space of K\"ahler metrics that are cohomologous to $\o$. $\mathcal H$ is a Fr\'echet manifold, and as such one can endow it with different Riemannian structures that induce path length metric spaces on $\mathcal H$. One such Riemannian structure was introduced by E. Calabi \cite{ca}  in 1954 (see also \cite{cl}), who proposed to find the metric completion of $\mathcal H$ using his metric. This was eventually carried out
by B. Clarke and Y. Rubinstein in \cite{cr}. As pointed out in \cite{cr}, it would be desirable, although much harder, to find the metric completion of $\mathcal H$ with respect to the much more studied Mabuchi metric. The main theorem of this paper answers this question completely, confirming a conjecture of V. Guedj \cite{g} in the process.

\subsection{Background}

To state our main result it is necessary to recall some facts about the Mabuchi geometry of $\mathcal H$. For a more exhaustive survey  we refer to \cite{bl}. For $v \in \mathcal{H}$ one can identify $T_v \mathcal{H}$ with $C^{\infty}(X)$. After Mabuchi \cite{m}, we introduce the following Riemannian metric on $\mathcal{H}$:
\[\langle \xi,  \eta \rangle_v := \int_X \xi \eta (\o + i\partial \overline{\partial}v)^n, \ \ \ \ \xi,\eta \in T_v \mathcal{H}. \]

%One can show that a smooth curve $(0,1)\ni t \to u_t \in \mathcal{H}$ is a geodesic in this space if
%$$\ddot u_t - \frac{1}{2} \langle \nabla\dot u_t, \nabla\dot u_t\rangle_{u_t}=0, \ t \in (0,1),$$
%where the dots denote differentiation with respect to $t$ and $\nabla\dot u_t$ is the gradient of $\dot u_t$ with respect to the metric $\o +i\partial\bar\partial u_t$. 
As discovered independently by Semmes \cite{s} and Donaldson \cite{do}, the geodesic equation of this metric can be written as a complex Monge-Amp\`ere equation. Let $(0,1)\ni t \to u_t \in \mathcal{H}$ be a smooth curve, $S = \left\{s \in \Bbb C: 0 <\textup{Re }s<1 \right\}$ and let $\tilde \o$ be the pullback of the K\"{a}hler form $\o$ to the product $S \times X$. Let $u \in C^{\infty}({S} \times X)$ be the complexification of $t \to u_t$, defined by $u(s,x) := u_{\textup{Re}s}(x)$. Then $t \to u_t$ is a geodesic if and only if the following equation is satisfied for $u$:
\begin{equation}\label{geodesic_eq}
(\tilde \o + i \partial_{S \times X} \overline{\partial}_{S \times X}u)^{n+1}=0.
\end{equation}
In the future we will not make a difference between a curve and its complexification.

Unfortunately, given $u_0,u_1 \in \mathcal H$ there may be no smooth geodesic connecting them (\cite{lv,dl,d2}). As shown by  X. X. Chen \cite{c} (with complements by Z. Blocki \cite{bli}), there exists certain weak geodesics connecting points of $\mathcal H$. To define these curves, one is driven into the world of pluripotential theory. The set of $\o$-plurisubharmonic functions is a natural generalization of $\mathcal H$:
$$\textup{PSH}(X,\o)=\{ u \equiv -\infty \hbox{ or }u \in C^{\uparrow}(X) \hbox{ with } \o + i\partial\bar\partial u \geq 0 \hbox{ in the sense of currents}\},$$
where $C^\uparrow(X)$ is the set of upper semi-continuous integrable functions on $X$. Chen proved that there exists a curve
$$[0,1]\ni t \to u_t \in \mathcal H_{\Delta}=\{ u \in \text{PSH}(X,\o)| \ \Delta u \in L^\infty(X)\}$$
connecting $u_0,u_1$ such that the map $(t,x) \to u_t(x)$ has bounded Laplacian, and the complexification of $u$ (again denoted by $u$) satisfies (\ref{geodesic_eq}) in the weak sense of Bedford-Taylor \cite{c}. For this curve we do not have ${u}_t \in \mathcal H$, $t \in (0,1)$, however many of its properties mimic those of actual geodesics (Section 2.1, see also \cite{cc,ct}).

As usual, the length  of a smooth curve $[0,1]\ni t \to \alpha_t \in \mathcal{H}$ is computed by the formula:
\begin{equation}\label{curve_length_def}
l(\alpha)=\int_0^1\sqrt{\int_X\dot \alpha_t^2(\o + i\partial \overline{\partial}\alpha_t)^n}dt.
\end{equation}
The distance $d({u_0},{u_1})$ between ${u_0},{u_1} \in \mathcal H$ is the infimum of the length of smooth curves joining ${u_0}$ and ${u_1}$. Chen proved that $d(u_0,u_1)=0$ if and only if $u_0 = u_1$, thus $(\mathcal H, d)$ is a metric space \cite{c}.

A curve $ (0,1) \ni t \to v_t \in \text{PSH}(X,\o)$ is called a subgeodesic segment if its complexification $v: S \times X \to \Bbb R$  is an $\o$-plurisubharmonic function. If additionally $v$ is locally bounded and satisfies (\ref{geodesic_eq}) in the sense of Bedford-Taylor, then $t \to v_t$ is called a weak geodesic segment.

\subsection{The Main Result}

One of the goals of this work is to generalize the notion of weak geodesic, so that given $u_0,u_1 \in \text{PSH}(X,\o)$ there is a weak geodesic $(0,1) \ni t \to u_t \in \text{PSH}(X,\o)$ connecting $u_0$ with $u_1$, i.e. $\lim_{t \to 0,1}=u_{0,1}$ in some suitable sense explained below. As we will see, this is not always possible, but by the end of this paper we partially attain this goal (Theorem \ref{geod_constr}).

When $u_0,u_1$ are bounded, Berndtsson constructed such weak geodesic $t \to u_t$ (Section 2.2 and \cite{br}). In the unbounded case we define $t\to u_t$ as a decreasing limit of bounded weak geodesics. Let $\{u^k_0\}_{k \in \Bbb N},\{u^k_1 \}_{k \in \Bbb N}\subset \mathcal H$ be sequences decreasing to $u_0$ and $u_1$ respectively. By \cite{bk} it is always possible to find such an approximating sequence.  Let $u^k_t : [0,1] \to \mathcal H_{\Delta}$ be the weak geodesic joining $u^k_0,u^k_1$, whose existence was proved by Chen. We define $t\to u_t$ as the decreasing limit:
\begin{equation}\label{udef}
u_t = \lim_{k \to + \infty}u^k_t, \ t \in (0,1).
\end{equation}
The curve $(0,1) \ni t \to u_t \in \text{PSH}(X,\o)$ we just constructed may be identically equal to $-\infty$, however one can see  that it is independent of the choice of approximating sequences, hence well defined. Indeed, $u$ is the upper envelope of a family $\mathcal S$ of subgeodesics:
$$u= \sup_{v \in \mathcal S}v,$$
where $ \mathcal S = \{ (0,1) \ni t \to v_t \in \text{PSH}(X,\o) \textup{ is a subgeodesic with }\lim_{t \to 0,1}v_t \leq u_{0,1} \}.$

One would also like to extend the path length metric $d$ on $\mathcal H$ to $\text{PSH}(X,\o)$ the obvious way, i.e. $\tilde d (u_0,u_1)=\lim_{k\to\infty}{d(u^k_0,u^k_1)}$, however it is not clear if this limit even exists or if it is independent of the approximating sequences. As we shall see shortly, to make this precise one has to restrict the definition of $\tilde d$ to a  subset of $\text{PSH}(X,\o)$. Before we can identify this domain, we need to recall some facts about finite energy classes on compact K\"ahler manifolds introduced by Guedj-Zeriahi \cite{gz}. For $v \in \text{PSH}(X,\o)$ let $v_h = \max \{ v,-h\}$, $h \in \Bbb R$. By an application of the comparison principle, it follows that the Borel measures $\mathbbm{1}_{\{v >-h\}}(\o + i\partial \bar\partial v_h)^n$
are increasing for $h >0$. By definition, $v \in \mathcal E(X,\o)$ if
\begin{equation*}
\lim_{h \to \infty} \int_X \mathbbm{1}_{\{v >-h\}}(\o + i\partial \bar\partial v_h)^n =\int_X \o^n=:\text{Vol}(X).
\end{equation*}
Suppose $\chi : \Bbb R \cup \{ -\infty\} \to \Bbb R$ is a continuous increasing function, with $\chi(0)= 0$ and $\chi(-\infty)=-\infty$. Such $\chi$ is referred to as a weight. The set of all weights is denoted by $\mathcal W$. Given $v \in \mathcal E(X,\o)$ we have $v \in \mathcal E_\chi(X,\o)$ if
$$E_\chi(v)=\lim_{h \to \infty}\int_{\{ v > -h\}} \chi(v_h) (\o + i\partial\bar\partial v_h)^n > -\infty.$$

We will be most interested in $\mathcal E^2(X,\o)$, this being the finite energy class given by the weight $\chi(t)= - (-t)^2$ (see Section 2.3). Given $u_0,u_1 \in \mathcal E^2(X,\o)$ and decreasing approximating sequences $u^k_0, u^k_1 \in \mathcal H$, we define $\tilde d(u_0,u_1)$, as promised, by the formula:
\begin{equation}\label{dist_geod_new}
\tilde d(u_0,u_1)=\lim_{k\to \infty}d(u^k_0,u^k_1).
\end{equation}
We will prove that this definition is independent of the choice of decreasing approximating sequences and $\tilde d$ is a metric on $\mathcal E^2(X,\o)$. Moreover, for $t \to u_t$ as defined in \eqref{udef}, we have $u_t \in \mathcal E^2(X,\omega), t \in (0,1)$ and the following theorem holds, which is our main result:
\begin{main_theorem}[Theorem \ref{e2space}, Theorem \ref{E2complete}]$(\mathcal E^2(X,\o), \tilde d)$ is a complete non-positively curved geodesic metric space, with geodesic segments joining $u_0,u_1 \in \mathcal E^2(X,\o)$ given by \eqref{udef}. Furthermore, $(\mathcal E^2(X,\o), \tilde d)$ can be identified with the metric completion of $(\mathcal H,d)$. \label{e2theorem}
\end{main_theorem}

This result was proved for toric K\"ahler manifolds and conjectured to hold in general by V. Guedj in a preliminary version of \cite{g}.

Let us recall that a geodesic metric space $(M,\rho)$ is a metric space for which any two points can be connected with a geodesic. By a geodesic connecting two points $a,b \in M$ we understand a curve $\alpha: [0,1]\to M$ such that $\alpha(0)=a$, $\alpha(1)=b$ and
\begin{equation}\label{geod_def}
\rho(\alpha(t_1),\alpha(t_2))=|t_1 -t_2|\rho(a,b),
\end{equation}
for any $t_1,t_2 \in [0,1]$.
Furthermore, a geodesic metric space $(M,\rho)$ is non-positively curved (in the sense of Alexandrov) if for any distinct points $q,r \in M$ there exists a geodesic $\gamma:[0,1] \to M$ joining $q,r$ such that for any $s \in \{ \gamma \}$  and $p \in M$ the following inequality is satisfied:
\begin{equation}\label{nonpositive}
\rho(p,s)^2 \leq \lambda \rho(p,r)^2 +(1-\lambda) \rho(p,q)^2 - \lambda(1-\lambda)\rho(q,r)^2,
\end{equation}
where $\lambda = \rho(q,s)/ \rho(q,r)$. A basic property of non-positively curved metric spaces is that geodesic segments joining different points are unique. For more about these spaces we refer to \cite{bh}.

As a consequence of the last theorem, using the estimates of Berndtsson \cite {br}, B\l ocki \cite{bl2}, \cite[Corollary 4.7]{bd} (see also \cite{h}) and Theorem \ref{geod_constr}, we obtain that $(\mathcal E^2(X,\o), \tilde d)$ has many special totally geodesic dense subspaces, with this answering positively questions raised in \cite{h}:

\begin{main_corollary}[Corollary \ref{supspaces}]The following sets are totally geodesic dense subspaces of $(\mathcal E^2(X,\o), \tilde d)$:
\begin{itemize}
\item[(i)] $\mathcal H_{\Delta}=\{ u \in \textup{PSH}(X,\o)| \ \Delta u \in L^\infty(X)\}$;
\item[(ii)] $\mathcal H_{0,1}= \textup{PSH}(X,\o) \cap \textup{Lip}(X)$;
\item[(iii)] $\mathcal H_0= \textup{PSH}(X,\o) \cap L^\infty(X)$;
%\item[(iv)] $\mathcal E_\chi(X,\o)$, if $ \chi \in \mathcal W^+_M, M \geq 1$ satisfies $\limsup_{t \to -\infty} \chi(t)/t^2 \leq -1$.
\end{itemize}
\end{main_corollary}

\subsection{Further Results}

The proof of Theorem \ref{e2theorem} involves the study of certain upper envelopes. The tools we develop will allow us to treat weak geodesics in a very general context, even beyond the class $\mathcal E^2(X,\o)$, which will be further explored in the paper \cite{d3}. Given $b_0,b_1 \in C^\uparrow(X)$, one can define the envelopes
$$P(b_0) = \sup \{ \psi \leq b_0: \psi \in \text{PSH}(X,\o)\},$$
$$P(b_0,b_1)= P(\min ( b_0,b_1)).$$
As the upper semi-continuous regularization $\textup{usc}(P(b_0))$ is an element of  $\{ \psi \leq b_0: \psi \in \text{PSH}(X,\o)\}$  and $P(b_0) \leq \textup{usc}(P(b_0))$, it follows that $P(b_0) \in \text{PSH}(X,\o)$ and similarly $P(b_0,b_1)\in \text{PSH}(X,\o)$. The motivation for the study of envelopes of the type $P(b_0,b_1)$ came from the following identity found in \cite{dr} (see also Section 2.2):
\begin{equation}
\inf_{t\in (0,1)} u_t = P(u_0,u_1),\label{intrlegenv}
\end{equation}
where $u_0,u_1 \in \mathcal H_0 = \text{PSH}(X,\o)\cap L^\infty$ and $(0,1) \ni t \to u_t \in \mathcal H_0$ is the weak geodesic connecting them. Further evidence that properties of the envelope $P(u_0,u_1)$ are tied together with the metric structure of the space $\mathcal H$ is given by the 'Pythagorean' identity of Theorem \ref{pythagorean}:
$$\tilde d(u_0,u_1)^2 = \tilde d(u_0,P(u_0,u_1))^2 + \tilde d(P(u_0,u_1),u_1)^2.$$
Our very first result about these envelopes says that the classes $\mathcal E_\chi(X,\o)$ are closed under the operation $ (u_0,u_1) \to P(u_0,u_1)$.

\begin{main_theorem}[Theorem \ref{envexist}]\label{E_min} If $\chi \in \mathcal W$ and ${u_0},{u_1} \in \mathcal E_\chi(X,\o)$, then $P({u_0},{u_1}) \in \mathcal E_\chi(X,\o)$. More precisely, if $N \in \Bbb R$ is such that $\chi \circ {u_0}, \chi \circ {u_1} \leq N$, then we have:
\begin{equation}\label{en_est}
E_\chi(P({u_0},{u_1})) \geq E_\chi({u_0}) + E_\chi({u_1})-N\textup{Vol}(X).
\end{equation}
Furthermore, if $\chi \in \mathcal W^- \cup \mathcal W^+_M, M \geq 1$, then $\mathcal E_\chi(X,\o)$ is convex.
\end{main_theorem}

For the definition of the weights $\mathcal W^-, \mathcal W^+_M$ and further properties of finite energy classes we refer to Section 2.3.
Convexity of ${\mathcal E}_\chi (X,\o)$  was already established
\cite{gz} in case $\chi \in \mathcal W^+_M, \ M \geq 1$. The case $\chi \in \mathcal W^-$ was conjectured in \cite[Remark 2.16]{begz}.

Next we turn to a characterization of $\mathcal E(X,\o)$ in terms of singularity types. For $u,v \in \text{PSH}(X,\o)$ we say that $u$ and $v$ have the same singularity type if there exists $c \in (0,\infty)$ such that
$$u - c \leq v \leq u + c.$$
This induces an equivalence relation on $\text{PSH}(X,\o)$ and we denote by $[u]$ the class of a representative $u\in \text{PSH}(X,\o)$.

Suppose now that ${u_0},{u_1} \in \text{PSH}(X,\o)$. In \cite{rwn1} the envelope of ${u_0}$ with respect to the singularity type of ${u_1}$ was introduced in the following manner:
\begin{equation}\label{envdef}
P_{[{u_1}]}({u_0})= \textup{usc}\Big(\lim _{c \to \infty}P({u_1}+c,{u_0})\Big).
\end{equation}
If ${u_0}$ is bounded and $u_1 \not\equiv -\infty$ then $P_{[{u_1}]}({u_0}) \not\equiv-\infty$. As a consequence of our first result, we see that this is also true when ${u_0},{u_1} \in\mathcal E(X,\o)$.

The functions of $\mathcal E(X,\o)$ can be unbounded, but their singularities are mild. In particular, by Corollary 1.8 \cite{gz}, at any $x \in X$ the Lelong number of $v$ is zero. However, as noted in \cite{gz}, the converse is false.  Membership in $\mathcal E(X,\o)$  can nevertheless be characterized in terms of envelopes:

\begin{theorem*}[\cite{d}] Suppose ${u_0} \in \textup{PSH}(X,\o) \cap C(X)$ and ${u_1} \in \textup{PSH}(X,\o)$. Then ${u_1} \in \mathcal E(X, \o)$ if and only if
$$P_{[{u_1}]}({u_0})={u_0}.$$
\end{theorem*}
The condition $P_{[{u_1}]}({u_0})={u_0}$ basically says that the singularity of ${u_1}$ is so mild that it is undetectable under envelope construction (\ref{envdef}). It is also easily seen to imply that all Lelong numbers of ${u_1}$ are zero \cite{d}. A question that immediately arises is weather the technical condition ${u_0} \in C(X)$ can be removed in the above result. Our next theorem, for which we will find other uses as well, says that this is indeed the case, what is more, ${u_0}$ need not even be bounded.
\begin{main_theorem}[Theorem \ref{Echar}] \label{E_char1}Suppose ${u_0} \in \mathcal E(X,\o)$ and ${u_1} \in \textup{PSH}(X,\o)$. Then ${u_1} \in \mathcal E(X, \o)$ if and only if
$$P_{[{u_1}]}({u_0})={u_0}.$$
\end{main_theorem}

We return now to weak geodesics. Given $u_0,u_1 \in \text{PSH}(X,\o)$, for the weak geodesic $t \to u_t$ defined in \eqref{udef} it may easily happen that $u \equiv -\infty$. If this is not the case, one wonders in what sense $u_t$ approaches $u_0,u_1$ as $t \to 0,1$. To investigate these questions we revisit the notion of capacity. As introduced by Ko\l odziej \cite{k}, the Monge-Amp\`ere capacity of a Borel set $B \subset X$ is defined by the formula
$$\text{Cap}(B)=\sup\Big\{\int_B (\o + i\partial\bar\partial u)^n \Big| u \in \text{PSH}(X,\o), 0 < u < 1   \Big\}.$$

For a sequence $\{v_k\}_{k \in \Bbb N} \subset \text{PSH}(X,\o)$ we say that $v_k \to v \in \text{PSH}(X,\o)$ in capacity if for any $\varepsilon > 0$ we have
$$\lim_{k \to \infty}\text{Cap}\big(\big\{x \in X \big| |v_k(x) -v(x)| > \varepsilon\big\}\big) = 0.$$

We note that convergence in capacity is stronger then convergence in $L^1(X)$ and is perhaps the strongest notion of convergence for unbounded plurisubharmonic functions. For an extensive study of capacities on compact K\"ahler manifolds we refer to \cite{gz2}.

We say that $u_0,u_1 \in \text{PSH}(X,\o)$ can be connected with a weak geodesic if the curve $t \to u_t$ defined in (\ref{udef}) satisfies $u \not\equiv \infty$ and $\lim_{t \to 0,1} u_t = u_{0,1}$ in capacity. It is not readily clear what pairs of potentials $(u_0,u_1)$ can be joined by a weak geodesic. Our next result connects this issue with properties of the different types of envelopes we introduced.

\begin{main_theorem}[Theorem \ref{boundary_limit}] \label{geod_constr}Suppose $u_0,u_1 \in \textup{PSH}(X,\o)$. Then for the curve $t \to u_t$ defined in \eqref{udef} we have:
\begin{enumerate}
\item[(i)] $u \not\equiv -\infty$ if and only if $P(u_0,u_1) \not\equiv -\infty.$
\item[(ii)] $\lim_{t \to 0} u_t = u_0$ in capacity if and only if $P_{[u_1]}(u_0)= u_0$.
\item[(iii)] $\lim_{t \to 1} u_t = u_1$ in capacity if and only if $P_{[u_0]}(u_1) = u_1$.
\end{enumerate}
\end{main_theorem}

As a trivial consequence of the last theorem we observe that if $u_0$ and $u_1$ have the same singularity type, then $u_0$ and $u_1$ can be connected with a weak geodesic. However, using \eqref{intrlegenv}, a more precise result can be obtained when $u_0,u_1$ are elements of finite energy classes:

\begin{main_theorem}[Corollary \ref{energy_geod1}] \label{energy_geod}Suppose $\chi \in \mathcal W^- \cup\mathcal W^+_M, M\geq1$ and $u_0,u_1 \in \mathcal E_\chi(X,\o)$. Then for the curve $t \to u_t$ defined in \eqref{udef} we have:
\begin{enumerate}
\item[(i)] $u_t \in \mathcal E_\chi(X,\o)$ for all $t \in (0,1).$ More precisely, if $N \in \Bbb R$ is such that $\chi \circ u_0,\chi \circ u_1 \leq N$ then we have:
\begin{equation}\label{geod_energy}
    E_\chi(u_t) \geq C(E_\chi(u_0) + E_\chi(u_1) - N \textup{Vol}(X)),
\end{equation}
where $C$ only depends on $M$ and $\dim X$.
\item[(ii)] $\lim_{t \to t_0} u_t = u_{t_0}$ in capacity for any $t_0 \in [0,1]$.
\end{enumerate}
\end{main_theorem}

Given $u_0 \in \text{PSH}(X,\o)$, one might want to find all $u_1 \in \text{PSH}(X,\o)$ such that $u_0$ and $u_1$ can be connected with a weak geodesic. As a consequence of  Theorem \ref{E_char1} and Theorem \ref{geod_constr}, for $u_0 \in \mathcal E(X,\o)$ we can provide an answer:
\begin{main_corollary}[Corollary \ref{geod_join}] Suppose $u_0 \in \mathcal E(X,\o)$ and $u_1 \in \textup{PSH}(X,\o)$. Then $u_0$ can be connected to $u_1$ with a weak geodesic if and only if $u_1 \in \mathcal E(X,\o)$.
\end{main_corollary}

\paragraph{Further applications and possible future directions.} Building on the techniques of this paper, in \cite{d3} we explore the Orlicz--Finsler geometry of $\mathcal H$, that is intimately tied together with the finite energy classes $\mathcal E_{\chi}(X,\o)$ for  $ \chi \in \mathcal W^+_M$. In addition to this, we also prove that convergence with respect to the path length metric $\tilde d$ can be characterized using very concrete terms. Indeed, we show that there exists $C>1$ such that for any $u_0,u_1 \in \mathcal E^2(X,\o)$ we have:
$$\frac{1}{C}\tilde d(u_0,u_1)^2 \leq \int_X (u_0 - u_1)^2 \o_{u_0}^n + \int_X (u_0 - u_1)^2 \o_{u_1}^n \leq C\tilde d(u_0,u_1)^2.$$
Further applications of the techniques developed here are explored in \cite{dh}, where we carry out a divergence analysis of the K\"ahler--Ricci flow in terms of the metric $\tilde d$. We also  construct destabilizing geodesic rays weakly asymptotic to diverging K\"ahler--Ricci flow trajectories, with this partially verifying a folklore conjecture.

Convergence of metrics with respect to the Calabi metric is equivalent to the $L^1$ convergence of their volume densities \cite{cr}. Though seemingly unrelated, in a future publication we hope to compare the geometry/topology of the Calabi metric with that of the the Mabuchi metric, as proposed in \cite{cr}.

In this paper we study extensively the end point problem for geodesics segments inside the metric completion of $\mathcal H$. Following the sequence of works initiated in \cite{rz}, it would be also natural to investigate the analogous initial value problem in the general context of the metric completion.

\paragraph{Acknowledgements.}  I would like to thank L. Lempert and Y. Rubinstein for their suggestions on how to improve the presentation of the paper and for their support during the years. I would also like to thank M. Jonsson for useful discussions, in particular for noticing that Theorem \ref{E_min} answers a question raised in \cite{begz}. I also profited from discussions with V. Guedj, G.K. Misiolek, and G. Sz\'ekelyhidi. The fact that the metric completion of $\mathcal H$ is non-positively curved has been explored recently by J. Streets in connection with the Calabi flow (\cite{st1}, \cite{st2}). We thank him for bringing to our attention his interesting papers.

\paragraph{Organization.} In Section 2 we recall preliminary material about finite energy classes and geodesics that we will need the most. We prove Theorem \ref{E_min} in Section 3, Theorem \ref{E_char1} in Section 4 and Theorems \ref{geod_constr} and \ref{energy_geod} in Section 5. The proof of Theorem \ref{e2theorem}  is given in  Sections 6--9. Readers interested only in the metric completion problem, should read Section 3, then skip ahead to the proof of Theorem \ref{energy_geod}(i) and from here proceed to Sections 6--9.

\section{Preliminaries}

\subsection{Distances in the Metric Space $(\mathcal H,d)$}

We summarize some of the properties of the metric space $(\mathcal H,d)$ that we will need later. This short section is based  entirely on the findings of \cite{c}. As in the introduction, given a smooth curve $[0,1]\ni t \to \alpha_t \in \mathcal{H}$, we define its length by the formula:
\begin{equation}\label{curve_length_def1}
l(\alpha)=\int_0^1\sqrt{\int_X\dot \alpha_t^2(\o + i\partial \overline{\partial}\alpha_t)^n}dt.
\end{equation}
The distance $d(u_0,u_1)$ between two points $u_0,u_1 \in \mathcal H$ is defined as the infimum of the length of smooth curves joining $u_0,u_1$. According to \cite{c}, $d$ is a metric on $\mathcal H$.

As noted earlier, in general there is no geodesic joining $u_0,u_1$ in $\mathcal H$, we only have a weak geodesic segment $[0,1] \ni t \to u_t \in \mathcal H_\Delta = \{ u \in \text{PSH}(X,\o)| \ \Delta u \in L^\infty(X)\}$ solving (\ref{geodesic_eq}) in the Bedford-Taylor sense, having bounded Laplacian. To address this technical inconvenience, X. X. Chen introduced the notion of $\varepsilon-$geodesics. An $\varepsilon-$geodesic joining $u_0,u_1$ is a smooth curve $[0,1] \ni t \to u^\varepsilon_t \in \mathcal H$ connecting $u_0,u_1$, that satisfies the equation:
\begin{equation}\label{approx_geod}
(\ddot{u^\varepsilon_t} - \frac{1}{2} \langle \nabla\dot u^\varepsilon_t, \nabla\dot u^\varepsilon_t\rangle)(\o + i\partial\bar\partial u^\varepsilon_t)^n=\varepsilon \o^n, \ t \in (0,1).
\end{equation}
Chen proved the existence of such curve for any $\varepsilon >0$, along with the formula:
\begin{equation}\label{dist_apr}
d(u_0,u_1)=\lim_{\varepsilon \to 0} l(u^\varepsilon).
\end{equation}
It is also proved that there exists $C>0$ independent of $\varepsilon$ such that
\begin{equation}\label{epsest}
\| \Delta u^\varepsilon\|_{L^\infty([0,1]\times X)} \leq C.
\end{equation}
Still, one would like to relate $d(u_0,u_1)$ directly to the weak geodesic $[0,1] \ni t \to u_t \in \mathcal H_{\Delta}$ joining $u_0,u_1$. For this we analyze (\ref{dist_apr}) more closely. We have that
$$l(u^\varepsilon)=\int_0^1\sqrt{E^\varepsilon(t)}dt,$$
where $E^\varepsilon(t)=\int_X\dot {u_t^\varepsilon}^2(\o + i\partial \overline{\partial}u_t^\varepsilon)^n$. Using (\ref{approx_geod}), one can easily compute that
$$\frac{dE^\varepsilon(t)}{dt}=2\varepsilon \int  {\dot u^\varepsilon_t} \o^n,$$
hence $|\dot E^\varepsilon(t)| \leq 2 \varepsilon  \| {\dot u^\varepsilon_t}\|_{L^\infty(X)}, \ t \in [0,1]$. Estimate (\ref{epsest}) further implies that
$$|\dot E^\varepsilon(t)| \leq 2\varepsilon C, \ t \in [0,1].$$
It follows from this that for any $t_0 \in [0,1]$ we have $\big|l(u^\varepsilon)-\sqrt{E^\varepsilon(t_0)}\big|=\big|\int_0^1 (\sqrt{E^\varepsilon(t)}-\sqrt{E^\varepsilon(t_0)})dt\big|\leq \sqrt{2 \varepsilon C}.$ This coupled with (\ref{dist_apr}) implies that
$$d(u_0,u_1)=\sqrt{\lim_{\varepsilon \to 0}E^\varepsilon(t_0)}.$$
It follows from the comparison principle (\cite[Theorem 6.4]{bl}) that our $\varepsilon-$geodesics $u^\varepsilon$ increase to the weak geodesic $u$ joining $u_0,u_1$, in particular $u^\varepsilon_{t_0}$ increases to $u_{t_0}$. Bedford-Taylor theory implies now that $(\o +i\partial\bar\partial u^\varepsilon_{t_0})^n \rightarrow (\o +i\partial\bar\partial u_{t_0})^n$ in the weak sense of measures. By (\ref{epsest}), using the Arzel\`a-Ascoli theorem it follows that we can find a subsequence of $u^\varepsilon_{t_0}$ (again denoted $u^\varepsilon_{t_0}$) such that $C(X) \ni \dot u^\varepsilon_{t_0} \to \dot u_{t_0} \in C(X)$ uniformly. The last two statements imply the following formula, again from \cite{c}:
\begin{equation}\label{distgeod}
d(u_0,u_1)=\sqrt{\lim_{\varepsilon \to 0}E^\varepsilon(t_0)} = \sqrt{\int_X\dot {u}_{t_0}^2(\o + i\partial \overline{\partial}u_{t_0})^n}.
\end{equation}
In Section 7 we will revisit (\ref{distgeod}) in a more general setting.

\subsection{Bounded Weak Geodesic Segments and Envelopes}

X. X. Chen's notion of weak geodesic can be generalized to construct weak geodesic segments connecting points of $\mathcal H_0 = \text{PSH}(X,\o)\cap L^\infty(X)$. Following Berndtssson, we recall how this argument works.

As before, let $S \subset \Bbb C$ be the strip $\{ 0 < \textup{Re }s < 1\}$ and $\tilde \o$ be the pullback of $\o$ to the product $S \times X$. As argued in \cite[Section 2.1]{br}, for $u_0,u_1 \in \mathcal H_0$ the following Dirichlet problem has a unique solution:
\begin{alignat}{2}\label{bvp_Bern}
&u \in \text{PSH}(S\times X, \tilde\o)\cap L^\infty(S\times X)\nonumber\\
&(\tilde \o + i \partial \overline{\partial}u)^{n+1}=0 \nonumber\\
&u(t+ir,x) =u(t,x) \ \forall x \in X, t \in (0,1), r \in \Bbb R \\
&\lim_{t \to 0,1}u(t,x)=u_{0,1}(x),  \forall x \in X\nonumber.
\end{alignat}
Since the solution to this equation is invariant in the imaginary direction, we denote it by $(0,1) \ni t\to u_t \in \mathcal H_0$ and call it the weak geodesic segment joining $u_0$ and $u_1$. Unsurprisingly, when $u_0,u_1 \in \mathcal H$ this is the same curve as the weak geodesic $t \to u_t$ mentioned in the previous section. Fittingly, a curve $(0,1)\ni t \to v_t \in \text{PSH}(X,\o)$ is called a subgeodesic if $v(s,x):=v_{\textup{Re} s}(x) \in \text{PSH}(S\times X,\tilde \o)$.

As a reminder, let us mention that the solution $u$ is constructed as the upper envelope
\begin{equation}\label{udef1}
u = \sup_{v \in \mathcal S}v,
\end{equation}
where $\mathcal S$ is the following set of weak subgeodesics:
$$ \mathcal S = \{ (0,1) \ni t \to v_t \in \text{PSH}(X,\o) \textup{ is a subgeodesic with }\lim_{t \to 0,1}v_t \leq u_{0,1} \}.$$
Berndtsson also proved that  $u$ is Lipschitz in the $t-$variable:
\begin{equation}\label{dot_u_est}
\bigg\| \frac{\partial u}{\partial t} \bigg\|_{L^\infty(X)} \leq \| u_0 - u_1\|_{L^\infty(X)}.
\end{equation}

Now we introduce the Legendre transform $\Bbb R \ni \tau \to u^*_\tau \in \text{PSH}(X,\o)$ of the weak geodesic segment $t \to u_t$ defined in (\ref{udef1}):
$$u^*_\tau = \inf_{t \in (0,1)}(u_t - \tau t).$$
The fact that $u^*_\tau \in \text{PSH}(X,\o)$ is guaranteed by Kiselman's minimum principle. The following identity is a particular case of a formula discovered  in \cite{dr}, where some of its applications are studied:
\begin{equation}\label{legenv} u^*_\tau = P(u_0,u_1 - \tau)
\end{equation}
The proof of this identity is quite elementary. On the one hand $P(u_0,u_1 - \tau) \leq u^*_\tau$ follows from (\ref{udef1}), as the subgeodesic $ t \to P(u_0,u_1 - \tau)+t \tau$ is included in $\mathcal S$. On the other hand, we clearly have $u^*_\tau \leq \min (u_0,u_1 -\tau )$, hence the definition of $P(u_0,u_1-\tau)$ implies the direction $ u^*_\tau \leq P(u_0,u_1 - \tau)$.

The regularity of envelopes of the type $P(u_0,u_1)$ will play an important role in this work. This is studied in \cite{dr} and we mention here the following particular case of interest:

\begin{theorem}[\cite{dr}] \label{DR_reg} If $u_0,u_1 \in \mathcal H_\Delta$ then $P(u_0,u_1) \in \mathcal H_\Delta$.
\end{theorem}

A quick proof for this theorem can be given using the already existent regularity theory of weak geodesic segments  \cite[Remark 4.3]{dr}.

\begin{proposition} \label{MA_form} Given $u_0,u_1 \in \mathcal H_\Delta$ we introduce $\Lambda_{u_0} = \{ P(u_0,u_1)=u_0\}$ and  $\Lambda_{u_1} = \{ P(u_0,u_1)=u_1\}$. Then the following partition formula holds for the Monge-Amp\`ere measure of $P(u_0,u_1)$:
\begin{equation}\label{MA_partition}
(\o + i\partial\bar\partial P(u_0,u_1))^n= \mathbbm{1}_{\Lambda_{u_0}}(\o + i\partial\bar\partial u_0)^n + \mathbbm{1}_{\Lambda_{u_1} \setminus \Lambda_{u_0}}(\o + i\partial\bar\partial u_1)^n.
\end{equation}
\end{proposition}
\begin{proof}
From \cite[Corollary 9.2]{bt} it follows that $(\o + i\partial\bar\partial P(u_0,u_1))^n$ is concentrated on the coincidence set $\Lambda_{u_0} \cup \Lambda_{u_1}.$ Having bounded Laplacian implies that all second order partials of $P(u_0,u_1)$ are in any $L^p(X), \ p <\infty$. It follows from \cite[Chapter 7, Lemma 7.7]{gt} that on $\Lambda_{u_0}$ all the second order partials of $P(u_0,u_1)$ and $u_0$ agree a.e. and an analogous statement holds on $\Lambda_{u_1}$. Hence, using \cite[Proposition 2.1.6]{bl3}  one can write:
\begin{flalign*}
(\o + i\partial\bar\partial P(u_0,u_1))^n&=\mathbbm{1}_{\Lambda_{u_0} \cup \Lambda_{u_1}}(\o + i\partial\bar\partial P(u_0,u_1))^n \nonumber \\
&=\mathbbm{1}_{\Lambda_{u_0}}(\o + i\partial\bar\partial u_0)^n + \mathbbm{1}_{\Lambda_{u_1} \setminus \Lambda_{u_0}}(\o + i\partial\bar\partial u_1)^n,
\end{flalign*}
finishing the proof.
\end{proof}
The partition formula (\ref{MA_partition}) is at the core of most theorems in this work. Interestingly, it does not hold even in the slightly more general case $u_0,u_1 \in \mathcal H_{0,1}$. For a counterexample suppose $\dim X = 1$ and $g_x$ is the $\o-$Green function with pole at $x \in X$. Such function is characterized by the property $\int_X g_x \o =0$ and $\o + i\partial \bar \partial g_x = \delta_x$. We choose $u_0 = \max\{g_x,0\}$ and $u_1 =0$. In this case $P(u_0,u_1)=0, \ \Lambda_{u_0} = \{ g_x \leq 0\}$ and $\Lambda_{u_1} = X \setminus \Lambda_{u_0} \neq \emptyset$. As
$ \textup{Vol}(X) = \int_{\Lambda_{u_0}} (\o + i\partial \bar \partial u_0)^n = \int_X (\o + i\partial \bar \partial P(u_0,u_1))^n,$ it is seen that the right hand side of (\ref{MA_partition}) has total integral greater the the left hand side, hence they can not equal.

Despite these difficulties, a one-sided generalization of this formula is still possible (see Proposition \ref{MA_form_prop}).

As the function $t \to \int_{\{ u_0  \leq u_1+t\}} \o^n$ is increasing, by adding constants one can always arrange that $\Lambda_{u_0} \cap \Lambda_{u_1} \subset \{ u_0=u_1\}$ has zero Lebesgue measure. Using this and the previous proposition, we can write down the following observation:
\begin{remark} \label{MA_form_remark} Given $u_0,u_1 \in \mathcal H_\Delta$ for any $\tau \in \Bbb R$ outside a countable set we have that $\Lambda_{u_0} \cap \Lambda_{u_1+\tau}$ has Lebesgue measure zero, implying:
$$(\o + i\partial\bar\partial P(u_0,u_1+\tau))^n= \mathbbm{1}_{\Lambda_{u_0}}(\o + i\partial\bar\partial u_0)^n + \mathbbm{1}_{\Lambda_{u_1+\tau}}(\o + i\partial\bar\partial u_1)^n.$$
\end{remark}
For more details on the regularity theory and more geometric properties of weak geodesic segments we refer to \cite{cc},\cite{bl},\cite{ct},\cite{do},\cite{h},\cite{ps},\cite{sz}, to name only a few articles in a very fast expanding literature.

\subsection{$\mathcal{E}(X,\o)$ and Finite Energy Classes}

We recall here the most basic facts about the class $\mathcal E(X,\o) \subset \text{PSH}(X,\o)$.  Our treatment is very brief and we refer to \cite{gz} and \cite{begz} for a more complete picture. For $v \in \text{PSH}(X,\o)$, one can define the canonical cutoffs $v_h \in \mathcal H_0, \ h \in \Bbb R$ by the formula $v_h = \max (-h, v ).$ By an application of the comparison principle, it follows that the Borel measures
$$\mathbbm{1}_{\{v >-h\}}(\o + i\partial \bar\partial v_h)^n$$
are increasing in $h$.
Even in $v$ is unbounded, one can still make sense of $(\o + i\partial\bar\partial v)^n$ as the limit of these increasing measures:
\begin{equation}\label{non-plurip}(\o + i\partial\bar\partial v)^n= \lim_{h \to \infty} \mathbbm{1}_{\{v >-h\}}(\o + i\partial \bar\partial v_h)^n.
\end{equation}
With this definition, $(\o + i\partial\bar\partial v)^n$ is called the non-pluripolar Monge-Amp\`ere measure of $v$. It follows from (\ref{non-plurip}) that
$$\int_X (\o + i\partial\bar\partial v)^n \leq \int_{X}\o^n =\text{Vol}(X).$$
This brings us to the class $\mathcal E(X,\o)$. By definition, $v \in \mathcal E(X,\o)$ if
\begin{equation}\label{Eps_def}
\int_X (\o + i\partial\bar\partial v)^n=\lim_{h \to \infty} \int_X \mathbbm{1}_{\{v >-h\}}(\o + i\partial \bar\partial v_h)^n =\text{Vol}(X).
\end{equation}

As shown in \cite{gz}, one can think of $\mathcal E(X,\o)$ as $\o$-plurisubharmonic functions having finite weighted energy. Suppose $\chi : \Bbb R \cup \{ -\infty \}\to \Bbb R$ is a continuous increasing function, with $\chi(0)= 0$ and $\chi(-\infty)=-\infty$. Such $\chi$ is referred to as a weight. The set of all weights is denoted by $\mathcal W$. By definition, for $v \in \mathcal E(X,\o)$ we have $v \in \mathcal E_\chi(X,\o)$ if
$$E_\chi(v)=\int_X \chi(v) (\o + i\partial\bar\partial v)^n > -\infty.$$
The following result says that the $\chi-$energy $E_\chi$ can be computed using approximation by the canonical cutoffs:
\begin{proposition}\textup{\cite[Proposition 1.4]{gz}}\label{cutoff_aprox} Suppose $u \in \mathcal E(X,\o), \chi \in \mathcal W$  and $u_k = \max ( u,-k), k \in \Bbb N$. Then:
$$E_\chi(u)> -\infty \textup{ if and only if } \inf_{k}E_\chi(u_k)=\inf_{k}\int_X \chi(u_k)\mathcal (\o + i\partial \bar \partial u_k)^{n}>-\infty.$$
If the above condition holds then we additionally have $E_\chi(u)= \lim_{k \to \infty}E_\chi(u_k).$
\end{proposition}
\noindent The two special classes of weights that we will be most interested are:
\begin{flalign*}
\mathcal W^-&=\big\{\chi \in \mathcal W \big| \chi \textup{ is convex and }\chi(t)=t, \ t \geq 0 \big\},\\
\mathcal W^+_M&=\big\{\chi \in \mathcal W \big| \chi \textup{ is concave and }|t\chi'(t)| \leq M|\chi(t)|, \ t \leq 0\big\},\end{flalign*}
where $M \geq 1$. The interest in the convex weights $\mathcal W^-$ comes from the following fact \cite[Proposition 2.2]{gz}:
\begin{equation}\label{E_union}
\mathcal E(X,\o)=\Big\{ v \in \mathcal E_\chi(X,\o)\Big|\chi \in \mathcal W^-\Big\}.
\end{equation}
Of special importance are the weights $\chi^p(t)=-(-t)^p, t \leq 0, p >0$ and the associated classes $\mathcal E^p(X,\o)$. Observe that $\chi^p \in \mathcal W^-$ for $p \leq 1$ and $\chi^p \in \mathcal W^+_p$ for $p \geq 1$. The case $p=1$ class interpolates between convex and concave energy classes as it is most apparent that
$$\mathcal E_\nu(X,\o) \subset \mathcal E^1(X,\o) \subset \mathcal E_\chi(X,\o),$$
for any $\nu \in \mathcal W^+_M$ and $\chi \in \mathcal W^-$.
%Given $\chi \in \mathcal W^-$ and $u \in \mathcal E_\chi(X,\o)$, it is easy to see that $\varepsilon u,u+c \in \mathcal E_\chi(X,\o)$ for any $0\leq \varepsilon \leq 1$ and $c \in \Bbb R$. The analogous result also holds for $\chi \in \mathcal W^+_M$, using the estimate \cite[Lemma 3.7]{gz}:
%$$\varepsilon^{M}|\chi(t)|\leq |\chi(\varepsilon t)| \leq \varepsilon|\chi(t)|, \ t <-1.$$
The following result is sometimes called the ``fundamental estimate":

\begin{proposition} \textup{\cite[Lemma 2.3, Lemma 3.5]{gz}}\label{Energy_est} Suppose $\chi\in \mathcal W^- \cup \mathcal W^+_M, \ M \geq1$. If $u,v \in \mathcal E_\chi(X,\o)$ with $u\leq v \leq 0$ then $$E_\chi(u) \leq  C E_\chi(v),$$ where $C>0$ depends only on $M$ and the dimension of $X$.
\end{proposition}
If $\chi\in \mathcal W^- \cup \mathcal W^+_M, \ M \geq1$ then the $\chi$-energy has a very useful semi-continuity property:
\begin{proposition}\textup{\cite[Proposition 5.6]{gz}} \label{E_semicont} Suppose $\chi\in \mathcal W^- \cup \mathcal W^+_M, \ M \geq1$ and $\{u_j\} _{j \in \Bbb N} \subset \mathcal H_0$ is a sequence decreasing to $u \in \textup{PSH}(X,\o)$. If
$\inf_j E_\chi(u_j) > -\infty$ then $u \in \mathcal E_\chi(X,\o)$ and
$$E_\chi(u) \geq \liminf_{j\to \infty}E_\chi(u_j).$$
\end{proposition}
Using the canonical cutoffs, the last two results imply  the very important ``monotonicity property":
\begin{corollary}\label{Emonoton} Suppose $\chi\in \mathcal W^- \cup \mathcal W^+_M, \ M \geq1$. If $u \leq v$ and $u \in \mathcal E_\chi(X,\o)$ then $v \in \mathcal E_\chi(X,\o)$.
\end{corollary}

The usual continuity property of the Monge-Amp\`ere operator from Bedford-Taylor theory is also preserved in this more general setting:
\begin{proposition}\textup{\cite[Theorem 2.17]{begz} \label{MA_cont}} Suppose  $\{ v_k\}_{k \in \Bbb N} \subset \mathcal E(X,\o)$ decreases (increases  a.e.) to $v \in \mathcal E(X,\o)$. Then $(\o + i\partial\bar\partial v_k)^n \to(\o + i\partial\bar\partial v)^n$ weakly.
\end{proposition}

Lastly, let us mention the uniqueness theorem of S. Dinew. We will use this result multiple times.
\begin{theorem} \textup{\cite{di}\label{dinew}} Suppose that $u,v \in \mathcal E(X,\o)$ satisfy $(\o + i\partial\bar\partial u)^n=(\o + i\partial\bar\partial v)^n.$ Then $u -v$ is constant.
\end{theorem}

\section{The Operator $(u,v) \to P(u,v)$ on $\mathcal E(X,\o)$}

In this section we examine properties of the envelopes $P({u_0},{u_1})$ when ${u_0},{u_1}$ are from one the finite energy classes $\mathcal E_\chi(X,\o)$. Before we start dealing with the general case, let us establish some preliminary results generalizing  formulas at the end of Section 2.2. By Proposition \ref{MA_form}, if ${u_0},{u_1} \in \mathcal H$ then $(\o+i\partial\bar\partial P({u_0},{u_1})^n$ only charges the coincidence set $\{ P({u_0},{u_1})=\min ( {u_0},{u_1})\}$. It turns out that this statement extends to the case ${u_0},{u_1} \in \mathcal H_0 = \textup{PSH}(X,\o) \cap L^\infty(X)$:

\begin{lemma} \label{meas_conc} For ${u_0},{u_1} \in \mathcal H_0$ we have
$$(\o+i\partial\bar\partial P({u_0},{u_1}))^n(\{ P({u_0},{u_1})<\min ( {u_0},{u_1})\})=0.$$
\end{lemma}

\begin{proof}By \cite{bk} there exist decreasing sequences $u_1^j,u_0^j \in \mathcal H$ such that $u_0^j \searrow {u_0}$ and $u_1^j \searrow u_1$. By Proposition \ref{MA_form} this means that
$$\int_X (\min( u_0^j,u_1^j)-P(u_0^j,u_1^j))(\o +i\partial\bar\partial P(u_0^j,u_1^j))^n=0, \ j \in \Bbb N.$$
Since $\min( u_0^j,u_1^j)= u_0^j + u_1^j - \max(u_0^j,u_1^j )$, and $\{u_0^j\}_{j\in\Bbb N},\{u_1^j\}_{j\in\Bbb N},\{\max (u_0^j,u_1^j )\}_{j\in\Bbb N}$ form  decreasing sequences of uniformly bounded $\o-$plurisubharmonic functions, by Bedford-Taylor theory, we can take the limit in the above identity to obtain:
$$\int_X (\min( {u_0},{u_1})-P({u_0},{u_1}))(\o +i\partial\bar\partial P({u_0},{u_1}))^n=0.$$
From this the statement of the lemma follows.
\end{proof}
With the next result, we take another step in generalizing Proposition \ref{MA_form}:
\begin{lemma} For ${u_0},{u_1} \in \mathcal H_0$ we have
$$(\o+i\partial\bar\partial P({u_0},{u_1}))^n\leq \mathbbm{1}_{\{ {u_0} \leq {u_1} \}}(\o+i\partial\bar\partial {u_0})^n + \mathbbm{1}_{\{ {u_1} \leq {u_0} \}}(\o+i\partial\bar\partial {u_1})^n.$$
\end{lemma}

\begin{proof} Let $u_0^j,u_1^j$ be strictly decreasing approximating sequences, as in the proof of the previous result. After taking the weak limit of measures, from Proposition \ref{MA_form} it follows that
\begin{equation}\label{weak_meas_est}
(\o+i\partial\bar\partial P({u_0},{u_1}))^n \leq (\o+i\partial\bar\partial {u_0})^n + (\o+i\partial\bar\partial {u_1})^n
\end{equation}
We next prove that
\begin{equation}\label{strong_meas_est}
\mathbbm{1}_{\{ {u_0} < {u_1} \}}(\o+i\partial\bar\partial P({u_0},{u_1}))^n \leq \mathbbm{1}_{\{ {u_0} < {u_1} \}}(\o+i\partial\bar\partial {u_0})^n
\end{equation}
First, we prove this estimate for smooth ${u_1}$. We can additionally assume without loss of generality that $\{ u^0_j = u_1\}$ has Lebesgue measure zero $j \in \Bbb N$. In this case the set $\{ {u_0} < {u_1} \}$ is open, more precisely, it is the union of the increasing open sets $\{ u_0^j < {u_1} \}$. Let $\phi \in C^\infty(X)$ with $\textup{supp }\phi \subset \{ {u_0} < {u_1}\}$ and $\phi \geq0$. Clearly, there exists $j_0$ such that $\textup{supp }\phi \subset \{ u_0^j < {u_1}\}$ for all $j \geq j_0$. Using Proposition \ref{MA_form} this implies that

$$\int_{\{{u_0}<{u_1}\}} \phi ((\o + i\partial\bar\partial u_0^j)^n - (\o + i\partial\bar\partial P(u_0^j,{u_1}))^n)\geq0, \ j \geq j_0.$$
Letting $j \to \infty$ we arrive at
$$\int_{\{{u_0}<{u_1}\}} \phi ((\o + i\partial\bar\partial {u_0}) - (\o + i\partial\bar\partial P({u_0},{u_1}))^n)\geq0, \ j \geq j_0.$$
This implies (\ref{strong_meas_est}) in the case ${u_1}$ is smooth. Now we treat the general case. Observe that the set $\{ {u_0} < {u_1}\}$ is contained in the intersection of the decreasing open sets $\{{u_0} < u_1^j\}, \ j \in \Bbb N$. In particular, by what we just proved,
$$\mathbbm{1}_{\{ {u_0} < {u_1} \}}(\o+i\partial\bar\partial P({u_0},u_1^j))^n \leq \mathbbm{1}_{\{ {u_0} < {u_1} \}}(\o+i\partial\bar\partial {u_0})^n, \ j \in \Bbb N.$$
Since $P({u_0},u_1^j)$ decreases to $P({u_0},{u_1})$ and $\{ {u_0} < {u_1} \}$ is plurifine open, we can use \cite[Lemma 3.5]{rwn1} to conclude that
$$\mathbbm{1}_{\{ {u_0} < {u_1} \}}(\o+i\partial\bar\partial P({u_0},{u_1}))^n \leq \liminf_{j \to \infty}\mathbbm{1}_{\{ {u_0} < {u_1} \}}(\o+i\partial\bar\partial P({u_0},u_1^j))^n.$$
Putting the last two estimates together we obtain (\ref{strong_meas_est}). By (\ref{weak_meas_est}) and (\ref{strong_meas_est}) it results that:
\begin{flalign*}
(\o+i\partial\bar\partial P({u_0},{u_1}))^n&= (\mathbbm{1}_{\{ {u_0} < {u_1} \}} + \mathbbm{1}_{\{ {u_1} < {u_0} \}} + \mathbbm{1}_{\{ {u_0} = {u_1} \}})(\o+i\partial\bar\partial P({u_0},{u_1}))^n \\
&\leq\mathbbm{1}_{\{ {u_0} < {u_1} \}}(\o+i\partial\bar\partial {u_0})^n + \mathbbm{1}_{\{ {u_1} < {u_0} \}}(\o+i\partial\bar\partial {u_1})^n + \\
 &\ \ \ +\mathbbm{1}_{\{ {u_0} = {u_1} \}}((\o+i\partial\bar\partial {u_0})^n + (\o+i\partial\bar\partial {u_1})^n).
\end{flalign*}
After rearranging terms we obtain the desired estimate.
\end{proof}

From the last two results we obtain an estimate that can be seen as a generalization of the partition formula in Proposition \ref{MA_form}:
\begin{proposition}\label{MA_form_prop}For ${u_0},{u_1} \in \mathcal H_0$ we have
\begin{equation}\label{MA_form_generalized}
(\o+i\partial\bar\partial P({u_0},{u_1}))^n\leq \mathbbm{1}_{\{ P({u_0},{u_1}) = {u_0} \}}(\o+i\partial\bar\partial {u_0})^n + \mathbbm{1}_{\{ P({u_0},{u_1}) = {u_1} \}}(\o+i\partial\bar\partial {u_1})^n
\end{equation}
\end{proposition}
We now move on to proving energy estimates for envelopes of the type $P({u_0},{u_1})$. As above, we deal with the bounded case first:

\begin{lemma} \label{env_exist} Suppose $\chi \in \mathcal W$, ${u_0},{u_1} \in \mathcal H_0$ and $N \in \Bbb R$ are such that $\chi \circ {u_0}, \chi \circ {u_1} \leq N$. Then
\begin{equation}\label{e_est}
E_\chi(P({u_0},{u_1})) \geq E_\chi({u_0}) + E_\chi({u_1})-N\textup{Vol}(X).
\end{equation}
\end{lemma}

\begin{proof} Since $P({u_0},{u_1}) \leq \min( {u_0},{u_1})$, it follows that $\chi \circ P({u_0},{u_1}) \leq N$.
\begin{flalign}E_\chi&(P({u_0},{u_1}))-N\textup{Vol}(X) = \int_X (\chi (P({u_0},{u_1}))-N)(\o + i\partial\bar\partial P({u_0},{u_1}))^n \nonumber \\
&\geq \int_{\{P({u_0},{u_1}) = {u_0}\}}(\chi({u_0})-N)(\o + i\partial\bar\partial {u_0})^n + \int_{\{P({u_0},{u_1}) = {u_1}\}}(\chi({u_1})-N)(\o + i\partial\bar\partial {u_1})^n \nonumber\\
&\geq \int_X(\chi({u_0})-N)(\o + i\partial\bar\partial {u_0})^n + \int_X(\chi({u_1})-N)(\o + i\partial\bar\partial {u_1})^n \nonumber \\
& = E_\chi({u_0}) + E_\chi({u_1}) - 2N\textup{Vol}(X).\nonumber
\end{flalign}
where in the second line we have used (\ref{MA_form_generalized}).
\end{proof}

\begin{corollary} \label{ecor}For ${u_0},{u_1} \in \mathcal E(X,\o)$, we have $P({u_0},{u_1})\in \mathcal E(X,\o)$. More precisely, if ${u_0},{u_1} \in \mathcal E_\chi(X,\o)$ for some $\chi \in \mathcal W^-$, then $P({u_0},{u_1}) \in \mathcal E_\chi(X,\o)$.
\end{corollary}

\begin{proof} Suppose ${u_0},{u_1} \in \mathcal E_\chi(X,\o)$ for some $\chi \in \mathcal W^-$. Let $N = \max(\sup_X\chi\circ{u_0},\sup_X\chi\circ{u_0}).$ If we apply the previous lemma to the canonical cutoffs $u_0^k$,$u_1^k$ we obtain:
$$E_\chi(P(u_0^k,u_1^k))\geq E_\chi(u_0^k) + E_\chi(u_1^k) - N\textup{Vol}(X).$$
Since ${u_0},{u_1} \in \mathcal E_\chi(X,\o)$ it follows that
\begin{flalign*}
\liminf_{k \to \infty}E_\chi(P(u_0^k,u_1^k))\geq& \lim_{k\to\infty}E_\chi(u_0^k) + \lim_{k\to\infty}E_\chi(u_1^k)-N\textup{Vol}(X)\\
=&E_\chi({u_0})+E_\chi({u_1}) -N\textup{Vol}(X),
\end{flalign*}
where we have used Proposition \ref{cutoff_aprox}. Since $\chi \in \mathcal W^-$ and $\inf_k E_\chi(P(u_0^k,u_1^k))$ is bounded below, using Proposition \ref{E_semicont} it follows that $P({u_0},{u_1}) \in \mathcal E_\chi(X,\o)$.

If we start with ${u_0},{u_1} \in \mathcal E(X,\o)$, then by (\ref{E_union}) it follows that there exists $\chi_1,\chi_2 \in \mathcal W^-$ such that ${u_0} \in \mathcal E_{\chi_1}(X,\o)$ and ${u_1} \in \mathcal E_{\chi_2}(X,\o)$. Clearly $\chi = \max ( \chi_1,\chi_2) \in \mathcal W^-$ and ${u_0},{u_1} \in \mathcal E_\chi(X,\o)$. By our prior argument it follows that $P({u_0},{u_1}) \in \mathcal E_\chi(X,\o) \subset \mathcal E(X,\o),$  finishing the proof.
\end{proof}

Finally, we generalize the estimate (\ref{e_est}) for any $\chi \in \mathcal W$ and ${u_0},{u_1} \in \mathcal E_\chi(X,\o)$, with this proving Theorem \ref{E_min}:
\begin{theorem} \label{envexist} If $\chi \in \mathcal W$ and ${u_0},{u_1} \in \mathcal E_\chi(X,\o)$, then $P({u_0},{u_1}) \in \mathcal E_\chi(X,\o)$. More precisely, if $N \in \Bbb R$ is such that $\chi \circ {u_0}, \chi \circ {u_1} \leq N$, then we have:
\begin{equation}\label{en_est}
E_\chi(P({u_0},{u_1})) \geq E_\chi({u_0}) + E_\chi({u_1})-N\textup{Vol}(X).
\end{equation}
Furthermore, if $\chi \in \mathcal W^- \cup \mathcal W^+_M, M \geq 1$, then $\mathcal E_\chi(X,\o)$ is convex.
\end{theorem}
\begin{proof}
As both $E_\chi({u_0})$ and $E_\chi({u_1})$ are finite, we can use Lemma \ref{env_exist} and Proposition \ref{cutoff_aprox} to conclude:
\begin{equation} \label{firsteq}
\liminf_{k\to \infty} E_\chi(P(u_0^k,u_1^k))\geq E_\chi({u_0})+E_\chi({u_1})-N\textup{Vol}(X),
\end{equation}
where $u_0^k$, $u_1^k$ are the canonical cutoffs.

Fix $l \in \Bbb N$ and $v \in C(X)$ satisfying $v \geq \chi\circ P(u_0^l,u_1^l)$. As the sequence $P(u_0^k,u_1^k)$ decreases to  $P({u_0},{u_1})$ and $P({u_0},{u_1}) \in \mathcal E(X,\o)$ by the previous result, we have $(\o + i\partial\bar\partial P(u_0^k,u_1^k))^n \to(\o + i\partial\bar\partial P({u_0},{u_1}))^n$ weakly. It follows that
$$ \int_X v(\o + i\partial\bar\partial P({u_0},{u_1}))^n = \lim_{k \to \infty} \int_X v(\o + i\partial\bar\partial P({u_0^k},{u_1^k}))^n \geq \liminf_{k\to \infty} E_\chi(P(u_0^k,u_1^k)).$$
As $\chi \circ P(u_0^l,u_1^l)$ is upper semi-continuous this implies:
$$ \int_X \chi(P(u_0^l,u_1^l))(\o + i\partial\bar\partial P({u_0},{u_1}))^n\geq \liminf_{k\to \infty} E_\chi(P(u_0^k,u_1^k)).$$
Letting $l \to + \infty$, by the monotone convergence theorem one obtains
\begin{equation} \label{secondeq}E_\chi(P(u_0,u_1))\geq \liminf_{k\to \infty} E_\chi(P(u_0^k,u_1^k)).
\end{equation}
Putting (\ref{firsteq}) and (\ref{secondeq}) together, (\ref{en_est}) follows.
Finally, we have the estimate: $$P({u_0},{u_1}) \leq \min({u_0},{u_1}) \leq t {u_0} + (1-t){u_1}, \ t \in [0,1].$$
We just proved that $P({u_0},{u_1}) \in \mathcal E_\chi(X,\o)$. If $\chi \in \mathcal W^- \cup \mathcal W^+_M, M \geq 1$ then Corollary \ref{Emonoton} implies that $t {u_0} + (1-t){u_1} \in \mathcal E_\chi(X,\o), \ t \in [0,1]$.
\end{proof}

We conclude from the above proof that in fact $\mathcal E_\chi(X,\o)$ is convex for all weights $\chi$ that have the "monotonicity" property described in Corollary \ref{Emonoton}: if $u \in \mathcal E_\chi(X,\o), \ v \in \text{PSH}(X,\o)$ and $u \leq v$ then $v \in \mathcal E_\chi(X,\o)$.

Finally, we notice that for arbitrary ${u_0},{u_1} \in \text{PSH}(X,\o)$,  it may easily happen that $P({u_0},{u_1})\equiv-\infty$. If $X$ is $\Bbb CP^1$ with the Fubini-Study metric $\o_{FS}$ and ${u_0},{u_1}$ are $\o_{FS}-$Green functions with poles at different points, then clearly $P({u_0},{u_1})\equiv-\infty$.

\section{The Operator $u \to P_{[{u}]}(v)$ on $\mathcal E(X,\o)$}

In this short section we will prove Theorem \ref{E_char1}. Before we do this we need some preliminary results.

\begin{lemma} \label{weak_est}Suppose $\chi \in \mathcal W$ and ${u_0},{u_1} \in \mathcal E_\chi(X,\o)$ with ${u_0},{u_1} < 0$. For any $c >0$ and $\phi \in C^\infty(X)$ with $\phi \geq0$  we have
$$\int_X \phi (\o + i\partial\bar\partial P({u_0},{u_1} + c))^n \leq \int_X \phi (\o + i\partial\bar\partial {u_0})^n + \frac{E_\chi({u_1})\sup_X \phi}{\chi(-c)}.$$
\end{lemma}
\begin{proof} Suppose $u_0^j,u_1^j \in \mathcal H_0$ are the canonical cutoffs decreasing to $u_0,u_1$. It follows that $P(u_0^j,u_1^j + c)$ decreases to $P(u_0,u_1 + c)$. Theorem \ref{envexist} implies that $P(u_0,u_1 + c)\in\mathcal E_\chi(X,\o)$  hence we have
$$\int_X \phi (\o + i\partial\bar\partial P({u_0},{u_1} + c))^n = \lim _{j \to \infty}\int_X \phi (\o + i\partial\bar\partial P(u_0^j,u_1^j + c))^n.$$
Using formula (\ref{MA_form_generalized}) we can write:
\begin{flalign*}
\int_X \phi (\o + &i\partial\bar\partial P(u_0^j,u_1^j + c))^n\leq\\
&\leq \int_{\{P(u_0^j,u_1^j + c) = u_0^j\}}\phi(\o + i\partial\bar\partial u_0^j)^n + \int_{\{P(u_0^j,u_1^j + c) = u_1^j + c\}}\phi(\o + i\partial\bar\partial u_1^j)^n\\
&\leq \int_X \phi(\o + i\partial\bar\partial u_0^j)^n + \int_{\{u_1^j < -c\}}\phi(\o + i\partial\bar\partial u_1^j)^n\\
&\leq \int_X \phi(\o + i\partial\bar\partial u_0^j)^n +\frac{\sup_X \phi}{\chi(-c)} \int_{\{u_1^j < -c\}}\chi(u_1^j)(\o + i\partial\bar\partial u_1^j)^n\\
&\leq \int_X \phi(\o + i\partial\bar\partial u_0^j)^n +\frac{\sup_X \phi}{\chi(-c)} \int_X\chi(u_1^j)(\o + i\partial\bar\partial u_1^j)^n\\
&= \int_X \phi(\o + i\partial\bar\partial u_0^j)^n +\frac{\sup_X \phi}{\chi(-c)} E_\chi(u_1^j).
\end{flalign*}
Using Proposition \ref{cutoff_aprox}, after taking the limit $j \to \infty$ in the above estimate we obtain the statement of the lemma.
\end{proof}

\begin{lemma} \label{dif_lemma} Suppose we have $u,v \in \textup{PSH}(X,\o)$ satisfying $P_{[u]}(v) \not\equiv -\infty$. Then $\inf_{\{ P_{[u]}(v) > -\infty\}}(v-P_{[u]}(v))=0.$
\end{lemma}
\begin{proof} Suppose $\inf_{\{ P_{[u]}(v) > -\infty\}}(v-P_{[u]}(v))>0$. Then there exists $\varepsilon >0$ such that $P_{[u]}(v) \leq v -\varepsilon$. This implies that $P(v,u+k) \leq v - \varepsilon, \ k \in \Bbb N$. From this it results that  $P(v,u+k)=P(P(v,u+k),u+k) \leq P(v-\varepsilon,u+k)= P(v,u+k+\varepsilon) - \varepsilon, \ k \in \Bbb N$. Letting $k \to \infty$, we obtain $P_{[u]}(v)\leq P_{[u]}(v)-\varepsilon$, a contradiction.
\end{proof}

Now we turn to the proof of Theorem \ref{E_char1}:

\begin{theorem} \label{Echar} Suppose ${u_0} \in \mathcal E(X,\o)$ and ${u_1} \in \textup{PSH}(X,\o)$. Then ${u_1} \in \mathcal E(X, \o)$ if and only if
$$P_{[{u_1}]}({u_0})={u_0}.$$
\end{theorem}
\begin{proof}We can assume that $u_0,u_1 < 0$. Suppose ${u_1} \in \mathcal E(X,\o)$. As explained in the proof of Corollary \ref{ecor}, one can find $\chi \in \mathcal W^-$ such that ${u_0},{u_1} \in \mathcal E_\chi(X,\o)$. By Theorem \ref{envexist} we have $P({u_0},{u_1} +k) \in \mathcal E_\chi(X,\o), \ k \in \Bbb N$. By the continuity of the Monge-Amp\`ere operator under increasing limits taken within the class $\mathcal E(X,\o)$, we have $(\o + i\partial \bar \partial P({u_0},{u_1} + k))^n \to (\o + i\partial \bar \partial P_{[{u_1}]}({u_0}))^n$ weakly as $k \to \infty$. Using this, as $\chi(-\infty)=-\infty$, from Lemma \ref{weak_est} it follows that
$$(\o + i\partial \bar \partial P_{[{u_1}]}({u_0}))^n\leq(\o + i\partial \bar \partial {u_0})^n.$$
Since both ${u_0}$ and $P_{[{u_1}]}({u_0})$ are in $\mathcal E_\chi(X,\o)$, both of the measures above integrate to $\text{Vol}(X)$ over $X$. Hence, we actually have equality in the above estimate. Now S. Dinew's uniqueness result (Theorem \ref{dinew}) and Lemma \ref{dif_lemma} yields that $P_{[{u_1}]}({u_0})={u_0}$.

For the other direction we use the results of \cite{d}. To be more precise, let $\tilde u_0 \in \mathcal H$ such that $\tilde u_0 \geq {u_0}$ and $\tilde u_0 \geq {u_1}$. We will argue that the geodesic ray $t \to v(\tilde u_0,{u_1})_t$ constructed by the method of \cite{d} is constant equal to $\tilde u_0$, thus implying that ${u_1} \in \mathcal E(X,\o)$ \cite[Theorem 2(iii)]{d}.

Let us recall that $(0,\infty) \ni t \to v(\tilde u_0,{u_1})_t \in
\mathcal H_0$ is a decreasing weak geodesic ray that is constructed as the increasing limit of the weak geodesic segments joining $\tilde u_0$ with $\max({u_1},\tilde u_0-k)$ parameterized by the segment $(0,k)$ \cite[Section 4]{d}. We also know that $v_\infty := \lim_{t \to \infty}v(\tilde u_0,{u_1})_t  \geq {u_1}$.

From \cite[Proposition 5.1]{d} it follows that $P_{[v_\infty]}(\tilde u_0) = v_\infty$, hence
$${u_1}\leq P_{[{u_1}]}(\tilde u_0)\leq P_{[v_\infty]}(\tilde u_0) = v_\infty.$$
If follows from the  method of constructing the ray $t \to v(\tilde u_0,{u_1})_t$ that for any $\alpha \in \text{PSH}(X,\o)$ with ${u_1} \leq \alpha \leq v_\infty$ we have $v(\tilde u_0,{u_1}) = v(\tilde u_0,\alpha)$. Using this, by the last estimate we have
$$v(\tilde u_0,P_{[{u_1}]}(\tilde u_0))_t=v(\tilde u_0,{u_1})_t,$$
for any $t \in (0,\infty)$. Since $P_{[{u_1}]}({u_0}) \leq P_{[{u_1}]}(\tilde u_0)$, we have  $P_{[{u_1}]}(\tilde u_0)\in \mathcal E(X,\o)$, hence by \cite[Theorem 2(iii)]{d} it follows that $t \to v(\tilde u_0,P_{[{u_1}]}(\tilde u_0))_t$ is constant equal to $\tilde u_0$, implying that $ t \to v(\tilde u_0,{u_1})_t$ is constant as well. As mentioned in the beginning, applying \cite[Theorem 2(iii)]{d} again, we obtain that ${u_1} \in \mathcal E(X,\o)$.
\end{proof}

\section{Weak Geodesic Segments in $\text{PSH}(X,\o)$}

Recall that given $u_0,u_1 \in  \text{PSH}(X,\o)$ and decreasing approximating sequences $u^k_0, u^k_1 \in \mathcal H$, we define the "candidate" weak geodesic $(0,1) \ni t \to u_t \in \text{PSH}(X,\o)$ joining $u_0$ and $u_1$ by the formula:
\begin{equation}\label{udef_new}
u_t = \lim_{k \to + \infty}u^k_t, \ t \in (0,1),
\end{equation}
where $(0,1)\ni t \to u^k_t \in \mathcal H_\Delta$ are the weak geodesics joining $u^k_0, u^k_1$. We observe that this definition is independent of the choice of approximating sequences, as $u(s,x)=u_{ \textup{Re }s}(x) \in PSH(S \times X, \tilde \o)$ is the upper envelope of the family $\mathcal S$:
\begin{equation}\label{udef1_new}
u= \sup_{v \in \mathcal S}v,
\end{equation}
where $\mathcal S$ is the following set of subgeodesics:
$$ \mathcal S = \{ (0,1) \ni t \to v_t \in \text{PSH}(X,\o) \textup{ is a subgeodesic with }\lim_{t \to 0,1}v_t \leq u_{0,1} \}.$$

By (\ref{udef1}) and (\ref{udef1_new}) it is clear that when $u_0,u_1 \in \mathcal H_0$ the weak geodesic defined in Section 2.2 and (\ref{udef_new}) are the same. Before we prove Theorem \ref{geod_constr} we make and elementary observation about convex functions that will turn out to be quite useful.

\begin{lemma} \label{easy_lemma}Given a bounded convex function $f:(0,1) \to \Bbb R$ we have
$$\lim_{t \to 0}f(t) = \lim_{\tau \to - \infty }\inf_{t \in (0,1)}(f(t)-\tau t).$$
\end{lemma}

\begin{proof} The estimate $\lim_{t \to 0}f(t) \geq \lim_{\tau \to - \infty }\inf_{t \in (0,1)}(f(t)-\tau t)$ is clear. Now we deal with the reverse estimate. If $f'$ is bounded below on $(0,1/2]$ then we are done, since for negative enough $\tau$ the map $t \to f(t) -\tau t$ is increasing. If $f'$ is unbounded on $(0,1/2]$, then for any $\tau < \min\{f'(1/2),0\}$ there exists $t_\tau \in (0,1/2)$ such that $\inf_{t \in (0,1)}(f(t)-\tau t)=f(t_\tau) -\tau t_\tau \geq f(t_\tau).$ Clearly $t_\tau \to 0$ as $\tau \to -\infty$, hence $ \lim_{\tau \to - \infty }\inf_{t \in (0,1)}(f(t)-\tau t) \geq \lim_{t \to 0}f(t).$
\end{proof}

\begin{theorem} \label{boundary_limit}Suppose $u_0,u_1 \in \textup{PSH}(X,\o)$, $u_0,u_1 \not\equiv -\infty$. Then for the curve $t \to u_t$ defined in (\ref{udef_new}) we have:
\begin{enumerate}
\item[(i)] $u \not\equiv -\infty$ if and only if $P(u_0,u_1) \not\equiv -\infty.$
\item[(ii)] $\lim_{t \to 0} u_t = u_0$ in capacity if and only if $P_{[u_1]}(u_0)= u_0$.
\item[(iii)] $\lim_{t \to 1} u_t = u_1$ in capacity if and only if $P_{[u_0]}(u_1) = u_1$.
\end{enumerate}
\end{theorem}
\begin{proof} We can suppose throughout the proof that $u_0,u_1 \leq 0$. By approximating with a decreasing sequence of bounded weak geodesics, it is easily seen that formula (\ref{legenv}) also holds for our possibly unbounded weak geodesic $t \to u_t$, that is
\begin{equation}\label{geod_env_new}
P(u_0,u_1-\tau) = \inf_{t \in (0,1)} (u_t - \tau t),
\end{equation}
for all $\tau \in \Bbb R$. By convexity in the $t$-variable, $u \not\equiv-\infty$ if and only if $\inf_{t \in (0,1)} u_t \not\equiv -\infty$, which in turn is equivalent to $P(u_0,u_1) \not\equiv -\infty$. This proves (i).

Now we turn to the proof of (ii). We assume first that $P_{[u_1]}(u_0)=u_0$. Notice that, by formula (\ref{udef1_new}), for any $c \in \Bbb R$ we have
$$P(u_0,u_1 + c) - ct \leq u_t, \ t \in (0,1).$$
By convexity in the $t$ variable we can further write:
$$P(u_0,u_1 + c)-u_0 - ct \leq u_t - u_0 \leq t(u_1 - u_0), \ t \in (0,1).$$
This implies that:
$$\{ |u_t - u_0| > \varepsilon\} \subset \{ |P(u_0,u_1 +c)-u_0| > \varepsilon - ct\} \cup \{ |u_1 - u_0| > \varepsilon/t\},$$
for any $t \in (0,1)$. Since the capacity is subadditive we can write:
\begin{flalign*}
\lim_{ t \to 0} \text{Cap}&\{ |u_t - u_0| > \varepsilon\}\leq\\
&\leq \limsup_{ t \to 0} \text{Cap}\{ |P(u_0,u_1 +c)-u_0| > \varepsilon - ct\} + \limsup_{ t \to 0}\text{Cap}\{ |u_1 - u_0| > \varepsilon/t\}\\
&\leq \text{Cap}\{ |P(u_0,u_1 +c)-u_0| > \varepsilon/2\} + \limsup_{t \to 0 }\text{Cap}\{ |u_1 - u_0| > \varepsilon/t\}.
\end{flalign*}
The last limit is zero as we have:
\begin{flalign*}
\limsup_{t \to 0 }\text{Cap}\{ |u_1 - u_0| > \varepsilon/t\}&\leq \limsup_{t \to 0 }\text{Cap}\{ |u_1| + |u_0| > \varepsilon/t\}\\
&\leq\lim_{t \to 0 }(\text{Cap}\{ u_0 < -\varepsilon/2t\} + \text{Cap}\{ u_1 < -\varepsilon/2t\})=0.
\end{flalign*}
Summing up we have
$$\lim_{ t \to 0} \text{Cap}\{ |u_t - u_0| > \varepsilon\}\leq \text{Cap}\{ |P(u_0,u_1 +c)-u_0| > \varepsilon/2\}, \ c \in \Bbb R.$$
Our assumption implies that $P(u_0,u_1 + c)$ increases to $u_0$ outside a set of zero capacity zero as $c \to + \infty$ (negligible sets are pluripolar). It is well known (using quasi-continuity for instance) that this implies
$$\lim_{c \to + \infty}\text{Cap}\{ |P(u_0,u_1 +c)-u_0| > \varepsilon/2\}=0,$$
proving that $\lim_{t \to 0}u_t=u_0$ in capacity.

Now we prove the other direction. As noted earlier, $\lim_{t \to 0}u_t=u_0$ in capacity implies convergence in $L^1(X)$. Using convexity of $u$ in the $t$ variable again , we obtain that $u_t(x) \to u_0(x)$ for any $x \in X$ outside a set $E$ of Lebesgue measure $0$. Lemma \ref{easy_lemma} now implies that:
$$u_0(x)= \lim_{\tau \to - \infty }\inf_{t \in (0,1)}(u_t(x)-\tau t), \ x \in X \setminus E.$$
Additionally, formula (\ref{geod_env_new}) coupled with the fact that negligible sets are pluripolar tells us that for any $x \in X$ outside a set $C$ of capacity zero, we have
$$\lim_{\tau \to - \infty }\inf_{t \in (0,1)}(u_t(x)-\tau t)= \lim_{\tau \to - \infty }P(u_0, u_1 - \tau)(x) =P_{[u_1]}(u_0)(x).$$

Putting the last two formulas together we obtain that $u_0(x)= P_{[u_1]}(u_0)(x)$ for any $x \in X$ outside $E \cup C$. As $\textup{Cap}(C)=0$, it follows that $C$ has Lebesgue measure zero, implying that  $E \cup C$ has Lebesgue measure zero as well. From this  it results that $u_0 = P_{[u_1]}(u_0)$ globally, finishing the proof of (ii). The proof of part (iii) is carried out the same way.
\end{proof}

The next corollary can be extracted from the proof of the previous theorem.

\begin{corollary} Suppose that $u_0,u_1 \in \text{PSH}(X,\o)$ and for the curve $t \to u_t$ defined in (\ref{udef_new}) we have $u \neq -\infty$. The following are equivalent:
\begin{enumerate}
\item[(i)] $P_{[u_1]}(u_0)= u_0$.
\item[(ii)] $\lim_{t \to 0} u_t = u_0$ in capacity.
\item[(iii)] $\lim_{t \to 0} u_t = u_0$ in $L^1(X)$.
\item[(iv)] $\lim_{t \to 0} u_t(x) = u_0(x)$ for any $x \in X$ outside a set of Lebesgue measure zero.
\end{enumerate}
The analogous statement for limits at $t=1$ is also true.
\end{corollary}
\begin{proof} By the previous theorem (i) and (ii) are equivalent. Clearly, (ii) implies (iii), which in turn implies (iv), since $u$ is convex in the $t-$variable.
Finally, the direction (iv)$\to$(i) follows from the last part of the proof of the previous theorem.
\end{proof}

We say that $u_0,u_1 \in \text{PSH}(X,\o)$,  $u_0,u_1 \not\equiv -\infty$ can be connected with a weak geodesic if for the curve $t \to u_t$ defined in (\ref{udef_new}) we have $u \not\equiv -\infty$ and $\lim_{t \to 0,1} u_t = u_{0,1}$ in capacity.
By putting together Theorem \ref{envexist}, Theorem \ref{Echar} and Theorem \ref{boundary_limit} we obtain that for certain weights $\chi$, the  elements of $\mathcal E_\chi(X,\o)$ can always be connected with a weak geodesic segment passing through $\mathcal E_\chi(X,\o)$.

\begin{corollary}\label{energy_geod1}Suppose $\chi \in \mathcal W^- \cup\mathcal W^+_M, M\geq1$ and $u_0,u_1 \in \mathcal E_\chi(X,\o)$. Then for the curve $t \to u_t$ defined in (\ref{udef_new}) we have:
\begin{enumerate}
\item[(i)] $u_t \in \mathcal E_\chi(X,\o)$ for all $t \in (0,1).$ More precisely, if $N \in \Bbb R$ satisfies $\chi \circ u_0,\chi \circ u_1 \leq N$ then:
\begin{equation}\label{geod_energy}
    E_\chi(u_t) \geq C(E_\chi(u_0) + E_\chi(u_1) - N Vol(X)),
\end{equation}
where $C$  depends only on $M$ and $\dim X$.
\item[(ii)]$\lim_{t \to t_0} u_t = u_{t_0}$ in capacity for any $t_0 \in [0,1]$.
\end{enumerate}
\end{corollary}
\begin{proof} Formula (\ref{geod_env_new}) implies that $u_t \geq P(u_0,u_1)$ for any $t_0 \in [0,1]$. Part (i) follows now from Theorem \ref{envexist}, Proposition \ref{Energy_est} and Proposition \ref{Emonoton}.
When $t_0 =0$ or $t_0 =1$, part (ii) is a consequence of Theorem \ref{Echar} and Theorem \ref{boundary_limit}(ii)(iii).

Let $t_0 \in (0,1)$. By part (i) we have $u_{t_0} \in \mathcal E_\chi(X,\o)$. Suppose $(0,1) \ni t \to v_t \in \mathcal E_\chi(X,\o)$ is the weak geodesic segment joining $u_0$ and $u_{t_0}$ defined by (\ref{udef_new}). By (\ref{udef1_new}) it follows that $v_t = u_{t_0 t}, \ t \in (0,1)$. Using what we just proved, we obtain that $\lim_{t \nearrow t_0} u_t =\lim_{t \nearrow 1} v_t = u_{t_0}$ in capacity. One deals with the right limit similarly to conclude that $\lim_{t \to t_0} u_t = u_{t_0}$ in capacity.
\end{proof}
We remark that in concluding $u_t \in E_\chi(X,\o)$ for all $t \in (0,1)$ we only used the fact that $\chi$ has the ``monotonicity" property described in Proposition \ref{Emonoton}. Another application of Theorem \ref{boundary_limit} and Theorem \ref{Echar} yields the last result of this section.
\begin{corollary} \label{geod_join}Suppose $u_0 \in \mathcal E(X,\o)$ and $u_1 \in \textup{PSH}(X,\o)$. Then $u_0$ can be connected to $u_1$ with a weak geodesic if and only if $u_1 \in \mathcal E(X,\o)$.
\end{corollary}

\section{Extending the Mabuchi Metric to $\mathcal E^2(X,\o)$}

Given $u_0,u_1 \in \mathcal E^2(X,\o)$ and decreasing approximating sequences $u^k_0, u^k_1 \in \mathcal H$, we define the distance $\tilde d(u_0,u_1)$ by the formula:
\begin{equation}\label{dist_geod_new1}
\tilde d(u_0,u_1)=\lim_{k\to \infty}d(u^k_0,u^k_1),
\end{equation}
The main result of this section is the following:
\begin{theorem} \label{e2space}$(\mathcal E^2(X,\o), \tilde d)$ is a non-positively curved geodesic metric space extending $(\mathcal H,d)$, with geodesic segments joining $u_0,u_1 \in \mathcal E^2(X,\o)$ given by (\ref{udef_new}).
\end{theorem}

Proving completeness of $(\mathcal E^2(X,\o), \tilde d)$ is left to a later section (Theorem \ref{E2complete}). The proof of the above theorem will be split into a sequence of lemmas. Our first result is a well known estimate for the Mabuchi metric that will be used a lot:

\begin{lemma}[\cite{c}] \label{Mdist_est}Suppose $u,v \in \mathcal H$ with $u \leq v$. Then we have:
$$\int_X(v-u)^2(\o +i\partial\bar\partial v)^n \leq d(u,v)^2 \leq \int_X(v-u)^2(\o +i\partial\bar\partial u)^n$$
\end{lemma}

\begin{proof} Suppose $(0,1) \ni t \to w_t \in \mathcal H_\Delta$ is the weak geodesic segment joining $u$ and $v$. By (\ref{distgeod}) we have
$$d(u,v)=\sqrt{\int_X \dot w^2_0 (\o +i\partial\bar\partial u)^n}=\sqrt{\int_X \dot w^2_1 (\o +i\partial\bar\partial v)^n}.$$ Since $u \leq v$, we have that $u \leq w_t$, as follows from (\ref{udef1}). Since $(t,x) \to w_t(x)$ is convex in the $t-$variable, it results that $0 \leq \dot w_0 \leq v-u \leq \dot w_1$ and the lemma follows.
\end{proof}

\begin{lemma} Suppose $u \in \mathcal E^2(X,\o)$ and $\{ u_k\}_{k \in \Bbb N} \subset \mathcal H$ is a sequence decreasing to $u$. Then $d(u_l,u_k) \to 0$ as $l,k \to \infty$.\label{IntDistEst}
\end{lemma}

\begin{proof}
We can suppose that $l \leq k$. Then $u_k\leq u_l$, hence by the previous lemma we have:
$$d(u_l,u_k)^2 \leq \int_X(u_k-u_l)^2(\o +i\partial\bar\partial u_k)^n.$$
We clearly have $u - u_l, u_k - u_l \in \mathcal E^2(X,\o +i\partial\bar\partial u_l)$ and $u - u_l\leq u_k - u_l\leq0$. Hence, applying Proposition \ref{Energy_est} for the class $\mathcal E^2(X,\o + i\partial\bar\partial u_l)$ we obtain that
\begin{equation}\label{estimate}
d(u_l,u_k)^2\leq C\int_X(u-u_l)^2(\o +i\partial\bar\partial u)^n.
\end{equation}
As $u_l$ decreases to $u \in \mathcal E^2(X,\o)$, it follows from the dominated convergence theorem that $d(u_l,u_k) \to 0$ as $l,k \to \infty$.
\end{proof}

\begin{lemma} Given $u_0,u_1 \in \mathcal E^2(X,\o)$, the limit in (\ref{dist_geod_new1}) is finite and independent of the approximating sequences $u^k_0, u^k_1 \in \mathcal H$. Additionally, if $u_0,u_1 \in \mathcal H$, then $\tilde d(u_0,u_1)$ is equal to the Mabuchi distance $d(u_0,u_1)$.
\end{lemma}

\begin{proof} By the triangle inequality and Lemma \ref{IntDistEst} we can write:
$$|d(u^l_0,u^l_1)-d(u^k_0,u^k_1)| \leq d(u^l_0,u^k_0) +d(u^l_1,u^k_1) \to 0, \ l,k \to \infty, $$
proving that $d(u^k_0,u^k_1)$ is indeed convergent.

Now we prove that the limit in \eqref{dist_geod_new1} is independent of the choice of approximating sequences. Let $v^l_0, v^l_1 \in \mathcal H$ be another approximating sequence. By adding small constants if necessary, we can arrange that all the sequences $u^l_0, u^l_1$ respectively $v^l_0, v^l_1$ are strictly decreasing to $u_0,u_1$.

Fixing $k$ for the moment, the sequence $\{\max\{ u^{k+1}_0,v^j_0\}\}_{j \in \Bbb N}$ decreases pointwise to $u^{k+1}_0$. By Dini's lemma the convergence is uniform, hence there exists $j_k\in \Bbb N$ such that for any $j \geq j_k$ we have $v^j_0 < u^k_0$. By repeating the same argument we can also assume that $v^j_1 < u^k_1$ for any $j \geq j_k$. By the triangle inequality again
$$|d(u^j_0,u^j_1)-d(v^k_0,v^k_0)| \leq d(u^j_0,v^k_0) +d(u^j_1,v^k_1), \ j \geq j_k. $$
From (\ref{estimate}) it follows that for $k$ big enough the quantities $d(u^j_0,v^k_0)$, $d(u^j_1,v^k_1), \ j \geq j_k$ are arbitrarily small, hence $\tilde d(u_0,u_1)$ is independent of the choice of approximating sequences. We observe that this automatically implies that $\tilde d$ restricted to $\mathcal H$ is the Mabuchi metric. The triangle inequality for $\tilde d$ also follows.
\end{proof}

To conclude that $\tilde d$ is a metric on $\mathcal E^2(X,\o)$ all we need is that $\tilde d(u_0,u_1)=0$ implies $u_0 = u_1$. Before we prove this we make two elementary observations.

\begin{lemma}\label{sublevel_lemma}Suppose $u_0,u_1 \in \mathcal H_0$. Let $(0,1) \ni t \to u_t \in \mathcal H_0$ be the bounded weak geodesic joining $u_0$ and $u_1$. Then for any $\tau \in \Bbb R$ we have
$$\{ \dot u_0 \geq \tau \} = \{ P(u_0,u_1 - \tau)=u_0\}.$$
\end{lemma}
\begin{proof} By (\ref{legenv}) we have $\inf_{t \in [0,1]}(u_t - \tau)= P(u_0,u_1 - \tau).$ Given $x \in X$, it follows that $P(u_0,u_1 - \tau)(x)=u_0(x)$ if and only if $\inf_{t \in [0,1]}(u_t(x) - \tau)=u_0(x)$. The convexity in the $t$-variable implies that this last identity is equivalent to $\dot u_0(x) \geq \tau$.
\end{proof}
The next result will serve a purpose similar to Lemma \ref{dif_lemma}.

\begin{lemma} \label{dif_lemma2}Suppose $u, v \in \textup{PSH}(X,\o)$ are from the same singularity class. Then $P(u,v) \not\equiv -\infty$, moreover either $P(u,v)=u$ or $\inf_{\{ u > -\infty\}} (v - P(u,v)) =0$.
\end{lemma}
\begin{proof} As $u, v \in \text{PSH}(X,\o)$ are from the same singularity class it follows that for $c$ big enough $P(u,v) \geq u-c$. If $u \leq v$ then clearly $P(u,v)=u$.

Let us suppose that $u \not\leq v$ and define $l_0 = \sup \{l \in \Bbb R | u + l \leq v\}.$ From our assumptions it follows that $-\infty < l_0 < 0$. We have $u + l_0 \leq u$ and $u+l_0 \leq v$, hence $u+l_0 \leq P(u,v)$. By the definition of $l_0$ we also have $\inf_{\{ u > -\infty\}} (v - u - l_0) =0$. Since $u+l_0 \leq P(u,v)\leq v$, it follows that $\inf_{\{ u > -\infty\}} (v - P(u,v)) =0$.
\end{proof}

\begin{lemma} For $u_0,u_1 \in \mathcal E^2(X,\o)$ if $d(u_0,u_1)=0$ then $u_0=u_1$.
\end{lemma}
\begin{proof} Suppose $u^k_0,u^k_1 \in \mathcal H$ are strictly decreasing approximating sequences of $u_0,u_1$. We also fix $\varepsilon >0$ . Let $(0,1) \ni t \to u^k_t \in \mathcal H_\Delta$ be the weak geodesic segment joining $u^k_0,u^k_1$. Starting from (\ref{distgeod}), we have the following sequence of estimates:
\begin{flalign*}
d(u^k_0,u^k_1)^2 &= \int_X  \dot {u^k_0}^2 (\o + i\partial\bar\partial u^k_0)^n \geq \varepsilon^2 \int_{\{\dot u^k_0 \geq \varepsilon\}}  (\o + i\partial\bar\partial u^k_0)^n\\
&=\varepsilon^2 \int_{\{ P(u^k_0,u^k_1 - \varepsilon)=u^k_0\}}  (\o + i\partial\bar\partial u^k_0)^n,\\
\end{flalign*}
where in the last line we have used Lemma \ref{sublevel_lemma}. As $d(u^k_0,u^k_1) \to 0$ it follows that
\begin{equation}\label{meas_est}\int_{\{ P(u^k_0,u^k_1 - \varepsilon)=u^k_0\}}  (\o + i\partial\bar\partial u^k_0)^n \to 0
\end{equation}
as $k \to \infty$. Suppose $\phi \in C^\infty(X)$ with $\phi \geq 0$. By Proposition \ref{MA_form} we have:
\begin{flalign*}
\int_X \phi &(\o + i\partial\bar\partial P(u^k_0,u^k_1-\varepsilon))^n \leq \\
&\int_{\{P(u^k_0,u^k_1-\varepsilon) = u^k_0\}}\phi(\o + i\partial\bar\partial u^k_0)^n + \int_{\{P(u^k_0,u^k_1-\varepsilon) = u^k_1-\varepsilon\}}\phi(\o + i\partial\bar\partial u^k_1)^n,
\end{flalign*}
$k \in \Bbb N.$  Using (\ref{meas_est}) it results that
$$\lim_{k \to\infty}\int_X \phi (\o + i\partial\bar\partial P(u^k_0,u^k_1-\varepsilon))^n \leq \lim_{k \to \infty}\int_X \phi (\o + i\partial\bar\partial u^k_1)^n.$$
As $P(u_0,u_1-\varepsilon) \in \mathcal E^2(X,\o)$ and $P(u^k_0,u^k_1-\varepsilon) \searrow P(u_0,u_1-\varepsilon)$, we can conclude:
$$(\o + i\partial\bar\partial P(u_0,u_1-\varepsilon))^n \leq (\o + i\partial\bar\partial u_1)^n.$$
As in the proof of Theorem \ref{Echar}, we have equality in this last estimate, since both measures integrate to $\text{Vol}(X)$. Now S. Dinew's uniqueness theorem implies that
\begin{equation}\label{firstid}
P(u_0,u_1-\varepsilon) = u_1-\varepsilon - b_\varepsilon
\end{equation}
for some $b_\varepsilon \geq 0$. We will conclude soon that in fact $b_\varepsilon=0$, but for the moment let us observe that we can similarly deduce that
\begin{equation}\label{secondid}
P(u_0-\varepsilon,u_1) = u_0-\varepsilon - c_\varepsilon
\end{equation}
for some $c_\varepsilon \geq 0$. It follows from (\ref{firstid}) and (\ref{secondid}) that $u_0$ and $u_1$ have the same singularity type. Lemma \ref{dif_lemma2} applied to (\ref{firstid}) implies now that either $u_0 = P(u_0,u_1 - \varepsilon)=u_1-\varepsilon - b_\varepsilon$ or $\inf_{X \setminus \{ u_0 = -\infty\}} (u_1 - \varepsilon - P(u_0,u_1 - \varepsilon))=0$. The first case implies that $d(u_0,u_1)=\varepsilon + b_\varepsilon >0$, contradicting our assumption. Hence we have $\inf_{X \setminus \{ u_0 = -\infty\}} (u_1 - \varepsilon - P(u_0,u_1 - \varepsilon))=0$ and this implies that $b_\varepsilon=0$. We can similarly conclude that $c_\varepsilon=0$.

Summing up, we have proved that $P(u_0,u_1-\varepsilon) = u_1-\varepsilon$ and $ P(u_0-\varepsilon,u_1) = u_0-\varepsilon$. This implies that $u_1 \geq u_0-\varepsilon$ and $u_0 \geq u_1-\varepsilon$. As $\varepsilon >0$ can be chosen arbitrarily small, we conclude that $u_0=u_1$.\end{proof}

By Corollary \ref{energy_geod1} it follows that given $u_0,u_1\in \mathcal E^2(X,\o)$, for the weak geodesic segment $t \to u_t$ connecting  $u_0,u_1$ we have $u_t \in\mathcal E^2(X,\o), \ t \in (0,1)$. Next we show that this weak geodesic is an actual geodesic segment in $(\mathcal E^2(X,\o),\tilde d)$ in the sense of (\ref{geod_def}). Before we do this we need a technical lemma:

\begin{lemma} \label{geod_tangent_limit} Suppose $v_0,v_1 \in \mathcal H_0$ and $\{v^j_1 \}_{j \in \Bbb N}\subset \mathcal H_0$ is sequence decreasing to $v_1$. By $(0,1) \ni t \to v_t,v_t^j \in \mathcal H_0$ we denote the bounded weak geodesic segments connecting $v_0,v_1$ and $v_0, v^j_1$ respectively. As we have convexity in the $t-$variable, we can define $\dot v_0 = \lim_{t \to 0}(v_t - v_0)/t$ and $\dot v^j_0 = \lim_{t \to 0}(v^j_t - v_0)/t$. The following holds:
$$\lim_{j \to \infty}\int_X \dot {v^j_0}^2 (\o + i\partial\bar\partial v_0)^n = \int_X \dot {v_0}^2 (\o + i\partial\bar\partial v_0)^n.$$
\end{lemma}
\begin{proof} By (\ref{dot_u_est}) there exists $C >0$ such that $\|\dot v_0\|_{L^\infty(X)}, \|\dot v^j_0\|_{L^\infty(X)} \leq C$. We also have $v \leq v^j, \ j \in \Bbb N$ by the comparison principle. As we have convexity in the $t-$variable and all our weak geodesics share the same starting point, it also follows that $\dot v^j_0 \searrow \dot v_0$ pointwise. The conclusion of the lemma follows now from Lebesgue's dominated convergence theorem.
\end{proof}

\begin{lemma} Suppose $u_0,u_1 \in \mathcal E^2(X,\o)$ and $(0,1)\ni t \to u_t \in \mathcal E^2(X,\o)$ is the weak geodesic segment connecting $u_0,u_1$ defined in (\ref{udef_new}). Then $t\ \to u_t$ is a geodesic segment in $(\mathcal E^2(X,\o),\tilde d)$ in the sense of (\ref{geod_def}).
\end{lemma}
\begin{proof} We have $u_t \in \mathcal E^2(X,\o), \ t \in [0,1],$ as follows from Corollarly \ref{energy_geod1}. First we prove that
\begin{flalign}\label{geod_half}
 \tilde d(u_0,u_l)=l\tilde d(u_0,u_1),
\end{flalign}
for $l \in [0,1]$. Suppose $u^k_0,u^k_1 \in \mathcal H$ are strictly decreasing approximating sequences of $u_0,u_1$ and let $(0,1)\ni t \to u^k_t \in \mathcal H_\Delta$ be the decreasing sequence of weak geodesics connecting $u_0^k,u_1^k$. By definition, we have that $$\tilde d(u_0,u_1)=\lim_{k\to \infty}d(u^k_0,u^k_1)=\lim_{k\to \infty}\sqrt{\int_X \dot {u^k_0}^2 (\o + i\partial\bar\partial u^k_0)^n},$$
where we have used (\ref{distgeod}).

The geodesics segments $(0,1)\ni t \to u^k_t \in \mathcal H_\Delta$ are decreasing pointwise to $(0,1)\ni t \to u_t\in \mathcal E^2(X,\o)$. In particular, this implies that $u^k_l \searrow u_l$.
We want to find a decreasing sequence $v^k_l \in \mathcal H$ such that $u^k_l \leq v^k_l$, $v^k_l \searrow u_l$ and
\begin{equation}\label{approx_lim}
ld(u_0^k,u^k_1) - d(u^k_0, v^k_l)=l\sqrt{\int_X \dot {u^k_0}^2 (\o + i\partial\bar\partial u^k_0)^n} - \sqrt{\int_X \dot {v^k_0}^2 (\o + i\partial\bar\partial u^k_0)^n} \to 0,
\end{equation}
as $k \to \infty$, where $(0,1)\ni t \to v^k\in \mathcal H_\Delta$ is the weak geodesic segment connecting $u^k_0$ and $v^k_l$. By the definition of $\tilde d$, letting $k \to \infty$ in (\ref{approx_lim}) would give us (\ref{geod_half}).

Finding such sequence $v_l^k$ is always possible by an application of the previous lemma to $v_0 := u^k_0$ and $v_1 := u^k_l$, as we observe that the weak geodesic segment connecting $u^k_0$ and $u^k_l$ is exactly $t \to u_{lt}^k$.

Finally, we prove that for $t_1,t_2 \in [0,1], \ t_1 \leq t_2$ we have
\begin{equation}\label{geod_def1}
\tilde d(u_{t_1},u_{t_2})=(t_2 -t_1)\tilde d(u_0,u_1).
\end{equation}
Let $h_0 = u_{t_2}$ and $h_1 = u_0$. If $(0,1) \ni t \to h_t \in \mathcal E^2(X,\o)$ is the weak geodesic connecting $h_0,h_1$ as defined in (\ref{udef_new}), using (\ref{udef1_new}) one can easily see that $h_t = u_{t_2(1-t)}$. Applying (\ref{geod_half}) to $t \to h_t$ and $l = 1 - t_1/t_2$ we obtain
$$(1 - t_1/t_2)\tilde d(u_{t_2},u_0)=\tilde d(u_{t_2},u_{t_1}).$$ Now applying (\ref{geod_half}) for $t \to u_t$ and $l = t_2$ we have
$$\tilde d(u_0, u_{t_2})=t_2 \tilde d(u_0,u_1).$$ Putting these last two formulas together we obtain (\ref{geod_def1}), finishing the proof.
\end{proof}
The following one sided generalization of Lemma \ref{Mdist_est} will be very useful:
\begin{lemma} Suppose $u,v \in \mathcal E^2(X,\o)$ satisfying $u \leq v$. Then:
$$\tilde d(u,v)^2 \leq C \int_X(v-u)^2(\o +i\partial\bar\partial u)^n,$$
where $C>0$ only depends on $\dim X$.
\end{lemma}
\begin{proof}Let $u_k,v_k \in \mathcal H$ be sequences decreasing to $u,v$, with the additional property $u_k \leq v_k, \ k \geq1$. By Lemma \ref{Mdist_est} we have:
$$d(u_k,v_k)^2 \leq \int_X(u_k-v_k)^2(\o +i\partial\bar\partial u_k)^n.$$
We clearly have $u-v_k, u_k -v_k \in \mathcal E^2(X,\o +i\partial\bar\partial v_k)$ and $u-v_k\leq u_k -v_k \leq 0$. Applying Proposition \ref{Energy_est} to $\mathcal E^2(X,\o +i\partial\bar\partial v_k)$ we can conclude:
$$d(u_k,v_k)^2 \leq C \int_X(u-v_k)^2(\o +i\partial\bar\partial u)^n.$$
Letting $k \to \infty$, by the dominated convergence theorem we arrive at the desired estimate.
\end{proof}

We prove now that monotone sequences in $\mathcal E^2(X,\o)$ converge with respect to the Mabuchi metric:

\begin{proposition} \label{Mdist_est_gen}
If $\{w_k\}_{k \in \Bbb N} \subset \mathcal E^2(X,\o)$ decreases (increases a.e.) to $w \in \mathcal E^2(X,\o)$ then $\tilde d(w_k,w)\to 0$.
\end{proposition}
\begin{proof} When $w_k$ is decreasing to $w$ the result follows from the estimate we just proved and the dominated convergence theorem. Suppose $w_k$ increases a.e. to $w$. Again, using the previous Lemma we want to prove that
\begin{equation}\label{intest}
\int_X (w -w_k)^2 (\o +i\partial \bar\partial w_k)^n \to 0.
\end{equation}
The rest of the argument is adapted from \cite[Theorem 2.17]{begz}. We can additionally suppose that all the functions are negative. Let $w^L_k = \max\{w_k,-L\}, L \geq 0$ with $w^L$ defined similarly. If the $w_k$ are uniformly bounded then (\ref{intest}) can be seen using the quasicontinuity property of pluri-subharmonic functions.  Hence we are done if we can prove that the quantity
$$\Big|\int_X (w - w_k)^2 (\o +i\partial \bar \partial w_k)^n-\int_X (w^L - w_k^L)^2 (\o +i\partial \bar \partial w^L_k)^n\Big| $$
tends to zero uniformly as $L \to \infty$. Before we get into the estimates we observe that there exists $\chi \in \mathcal W^+_M$ for some $M > 1$ such that $-t^2/\chi(t)$ decreases to $0$ as $t \to -\infty$ and $w_k,w \in \mathcal E_\chi(X,\o), \ k \in \Bbb N$. We can write:
\begin{flalign*}
\Big|\int_X (w - w_k)^2 &(\o +i\partial \bar \partial w_k)^n-\int_X (w^L - w_k^L)^2 (\o +i\partial \bar \partial w^L_k)^n\Big| \leq \\
&\leq\Big|\int_{\{w_k \leq -L\}} (w - w_k)^2 (\o +i\partial \bar \partial w_k)^n-\int_{\{w_k \leq -L\}} (w^L - w_k^L)^2 (\o +i\partial \bar \partial w^L_k)^n\Big| \\
&\leq\Big(\int_{\{w_k \leq -L\}} w_k^2 (\o +i\partial \bar \partial w_k)^n+\int_{\{w_k \leq -L\}} {w_k^L}^2 (\o +i\partial \bar \partial w^L_k)^n\Big)\\
&\leq\frac{L^2}{\chi(L)}\Big(\int_{\{w_k \leq -L\}} \chi(w_k) (\o +i\partial \bar \partial w_k)^n+\int_{\{w_k \leq -L\}} \chi(w_k^L) (\o +i\partial \bar \partial w^L_k)^n\Big)\\
&\leq\frac{C L^2}{|\chi(L)|}|E_\chi(w_k)| \leq \frac{C L^2}{\chi(L)},
\end{flalign*}
where in the last line we have used Proposition \ref{Energy_est}.
\end{proof}

\begin{lemma} $(\mathcal E^2(X,\o), \tilde d)$ is non-positively curved in the sense of Alexandrov. In particular, goedesic segments joining different points of $\mathcal E^2(X,\o)$ are unique.
\end{lemma}
\begin{proof} Suppose $p,q,r \in \mathcal E^2(X,\o)$ and $(0,1) \ni t \to u^{qr}_t \in \mathcal E^2(X,\o)$ is the weak geodesic segment connecting $q,r$. Suppose $s \in \mathcal \{{u^{qr}}\}$ with $\lambda \tilde d(q,r)= \tilde d(q,s)$, i.e. $s = u^{qr}_\lambda$. We need to prove the estimate
\begin{equation}\label{nonpositive1}
\tilde d(p,s)^2 \leq \lambda \tilde d(p,r)^2 +(1-\lambda) \tilde d(p,q)^2 - \lambda(1-\lambda)\tilde d(q,r)^2.
\end{equation}
First, suppose $p,q,r \in \mathcal H$. The proof of the above inequality in this case  is done in \cite{cc}, so we refer to this work and only skim over the argument.

We fix $\varepsilon,\varepsilon' >0$. With $[0,1] \ni t \to u^{\varepsilon,pq}_t,u^{\varepsilon,pr}_t,u^{\varepsilon',qr}_t \in \mathcal H$ we denote the smooth $\varepsilon-$geodesics ($\varepsilon'-$geodesics) joining $(p,q),(p,r)$ and $(q,r)$ respectively (see Section 2.1). Given a smooth curve $[0,1]\ni t \to \alpha_t \in \mathcal H$, we denote by $L(\alpha)$ its energy:
$$L(\alpha)=\int_0^1\int_X \dot \alpha_t^2(\o+i\partial\bar\partial \alpha_t)^n.$$
Let $[0,1]\ni t\to u^{\varepsilon,\lambda}_t \in \mathcal H$ be the $\varepsilon-$geodesic joining $p$ and $u^{\varepsilon',qr}_\lambda$. It is proved in  \cite[Section 2.4]{cc} that
$$L(u^{\varepsilon,\lambda})\leq (1-\lambda)L(u^{\varepsilon,pq}) + \lambda L(u^{\varepsilon,pr}) -\lambda(1-\lambda) L(u^{\varepsilon',qr}) - (\varepsilon+\varepsilon') C$$
where $C >0$ is independent of $\varepsilon,\varepsilon'$. Using the same ideas that lead to $(\ref{distgeod})$, after letting first $\varepsilon \to 0$ then $\varepsilon' \to 0$ in the above estimate we arrive at
$$\tilde d(p,s)^2=\tilde d(p,u_\lambda)^2 \leq (1-\lambda)\tilde d(p,q)^2 + \lambda\tilde d(p,r)^2 -\lambda(1-\lambda)\tilde d(q,r)^2.$$
In general, if $p,q,r \in \mathcal E^2(X,\o)$, we can use approximation: suppose $p_k,q_k,r_k \in \mathcal H$ are sequences decreasing to $p,q,r$. From what we proved it follows that
$$\tilde d(p^k,u^{q^kr^k}_\lambda)^2 \leq (1-\lambda)\tilde d(p^k,q^k)^2 + \lambda\tilde d(p^k,r^k)^2 -\lambda(1-\lambda)\tilde d(q^k,r^k)^2,$$
where $(0,1)\ni t \to u^{q^kr^k}_t \in \mathcal H_\Delta$ is the weak geodesic connecting $q^k,r^k$. From the comparison principle it follows that $\{u^{q^kr^k}\}_k$ is decreasing to $u^{qr}$, in particular $u^{q^kr^k}_\lambda \searrow u^{qr}_\lambda=s$. Using Proposition \ref{Mdist_est_gen} we can take the limit in this last estimate to obtain (\ref{nonpositive1}).
\end{proof}
As a corollary to Theorem \ref{e2space}, using the estimates of Berndtsson (\ref{dot_u_est}), Blocki \cite[Theorem 1]{bl2}, He \cite[Theorem 1.1]{h} and Theorem \ref{geod_constr}, we obtain that $(\mathcal E^2(X,\o), \tilde d)$ has many special dense totally geodesic subspaces:
\begin{corollary}\label{supspaces} $\mathcal H$ is dense in $\mathcal E^2(X,\o)$ and the following sets are totally geodesic dense subspaces of $(\mathcal E^2(X,\o), \tilde d)$:
\begin{itemize}
\item[(i)] $\mathcal H_{\Delta}= \{ u \in \textup{PSH}(X,\o)| \ \Delta u \in L^\infty(X)\}$;
\item[(ii)] $\mathcal H_{0,1}= \textup{PSH}(X,\o) \cap \textup{Lip}(X)$;
\item[(iii)] $\mathcal H_0= \textup{PSH}(X,\o) \cap L^\infty(X)$;
\item[(iv)] $\mathcal E_\chi(X,\o)$, if $ \chi \in \mathcal W^+_M, M \geq 1$ satisfies $\limsup_{t \to -\infty} \chi(t)/t^2 \leq -1$.
\end{itemize}
\end{corollary}
\begin{proof} Suppose $\mathcal S$ is one of the above subspaces. From the above quoted results it follows that given $u_0,u_1 \in \mathcal S$, for the weak geodesic segment $t \to u_t$ joining $u_0,u_1$ we have $u_t \in \mathcal S, \ t \in (0,1)$.

As $\mathcal H\subset \mathcal S$, we just need to argue that $\mathcal H \subset \mathcal E^2(X,\o)$ is dense. As elements of $\mathcal E^2(X,\o)$ can be approximated by decreasing sequences in $\mathcal H$, an application of Proposition \ref{Mdist_est_gen} finishes the proof.
\end{proof}

\section{The Length of Geodesic Segments}

In light of results obtained in the previous section, one would like to generalize formula (\ref{distgeod}) for arbitrary $u_0,u_1 \in \mathcal E^2(X,\o)$. However this fails even for $u_0,u_1 \in \mathcal H_{0,1}$, as the next (familiar) example shows. Suppose $\dim X =1$ and $g_x \in \text{PSH}(X,\o)$ is the $\o-$Green function with pole at some $x \in X$. As in Section 2.2, let $u_0 = \max(g_x,0)$ and $u_1=0$. We denote by $(0,1) \ni t \to u_t \in \mathcal H_0$ the geodesic connecting $u_0$ and $u_1$. Clearly $\tilde d (u_0,u_1) >0$. By properties of Green functions, $(\o + i\partial \bar \partial u_0)$ only charges the set $\{g_x \leq 0\}$. However $\dot u_0\big|_{\{g_x \leq 0\}}\equiv0$, because $0\leq u_t, \ t \in [0,1]$ and $t \to u_t$ is decreasing. It results that
$$\int_X {\dot u^2_0}(\o + i\partial \bar \partial u_0)=0,$$
contradicting (\ref{distgeod}).
This example also suggests that one can not endow $\mathcal E^2(X,\o)$ with a Riemannian structure that would induce the metric space $(\mathcal E^2(X,\o),\tilde d)$ unlike $(\mathcal H, d)$.

The main purpose of this section is to show that (\ref{distgeod}) nevertheless  holds for $u_0,u_1 \in \mathcal H_\Delta$. Recall that by \cite[Corollary 4.7]{bd} (see also \cite[Theorem 1.1]{h}) we have $u_t \in \mathcal H_\Delta,\ t \in (0,1)$,  where $t \to u_t$ is the geodesic joining $u_0,u_1$. We start with a lemma:

\begin{lemma} Suppose $u_0,u_1 \in \mathcal H_\Delta$ and $(0,1) \ni t \to u_t \in \mathcal H_\Delta$ is the geodesic connecting them. Then the following holds:
\begin{equation}\label{lengthequal}
\int_X {\dot u^2_0}(\o + i\partial \bar \partial u_0)^n=\int_X {\dot u^2_1}(\o + i\partial \bar \partial u_1)^n.
\end{equation}
\end{lemma}
\begin{proof} To obtain (\ref{lengthequal}) we prove the following two formulas:
\begin{equation}\label{lengthequal1}\int_{\{\dot u_0 >0\}} {\dot u^2_0}(\o + i\partial \bar \partial u_0)^n=\int_{\{\dot u_1 > 0\}} {\dot u^2_1}(\o + i\partial \bar \partial u_1)^n,
\end{equation}
\begin{equation}\label{lengthequal2} \int_{\{\dot u_0 <0\}} {\dot u^2_0}(\o + i\partial \bar \partial u_0)^n=\int_{\{\dot u_1 <0\}} {\dot u^2_1}(\o + i\partial \bar \partial u_1)^n.
\end{equation}
Using Remark \ref{MA_form_remark} and Lemma \ref{sublevel_lemma} multiple times we can write:
\begin{flalign*} \int_{\{\dot u_0 > 0\}} \dot u_0^2 (\o + i\partial \bar \partial u_0)^n&= 2\int_0^{\infty} \tau(\o + i\partial \bar \partial u_0)^n(\{{\dot u_0 \geq \tau}\})d\tau\\
&= 2\int_0^{\infty} \tau(\o + i\partial\bar\partial u_0)^n(\{ P(u_0,u_1 - \tau)=u_0\})d\tau\\
&= 2\int_0^{\infty} \tau(\textup{Vol}(X)-(\o + i\partial\bar\partial u_1)^n(\{ P(u_0,u_1 - \tau)=u_1-\tau\}))d\tau\\
&= 2\int_0^{\infty} \tau(\o + i\partial\bar\partial u_1)^n(\{ P(u_0,u_1 - \tau)<u_1-\tau\})d\tau\\
&= 2\int_0^{\infty} \tau(\o + i\partial\bar\partial u_1)^n(\{ P(u_0+\tau,u_1)<u_1\})d\tau\\
&= 2\int_0^{\infty} \tau(\o + i\partial\bar\partial u_1)^n(\{ \dot u_1 > \tau\})d\tau\\
&=\int_{\{\dot u_1 > 0\}} {\dot u^2_1}(\o + i\partial \bar \partial u_1)^n,
\end{flalign*}
where in the second we have used Lemma \ref{sublevel_lemma}, in the third line we have used Remark \ref{MA_form_remark} and in the sixth line we have used Lemma \ref{sublevel_lemma} again. Formula (\ref{lengthequal2}) follows if we apply (\ref{lengthequal1}) to the ``reversed" geodesic $t \to v_t = u_{1-t}$.
\end{proof}

\begin{theorem} \label{distgeod_general} Suppose $u_0,u_1 \in \mathcal H_\Delta$ and $(0,1) \ni t \to u_t \in \mathcal H_\Delta$ is the geodesic connecting them. Then we have:
\begin{equation}\label{distgeod_formula}
\tilde d(u_0,u_1)^2 = \int_X {\dot u^2_t}(\o + i\partial \bar \partial u_t)^n, \  t \in [0,1].
\end{equation}
\end{theorem}
\begin{proof} As usual, let $u^k_0,u^k_1 \in \mathcal H$ be a sequence of potentials decreasing to $u_0,u_1$. Let $(0,1) \ni t \to u^{kl}_t \in \mathcal H_\Delta$ be the geodesic joining $u_0^k,u^l_1$. By (\ref{distgeod}) we have
$$\tilde d(u_0^k,u^l_1)^2 = \int_X {\dot {u^{kl}_0}^2}(\o + i\partial \bar \partial u_0^k)^n.$$
If we let $l \to \infty$, by Lemma \ref{geod_tangent_limit} and Proposition \ref{Mdist_est_gen} we obtain that that
$$\tilde d(u_0^k,u_1)^2 = \int_X {\dot {u^k_0}^2}(\o + i\partial \bar \partial u_0^k)^n,$$
where $(0,1) \ni t \to u^k_t \in \mathcal H_\Delta$ is the geodesic connecting $u^k_0$ with $u_1$. Using the previous lemma we can write:
$$\tilde d(u_0^k,u_1)^2 = \int_X {\dot {u^k_1}^2}(\o + i\partial \bar \partial u_1)^n.$$
Letting $k \to \infty$, another application of Lemma \ref{geod_tangent_limit} yields (\ref{distgeod_formula}) for $t=1$. The case $t=0$ follows by symmetry, and for $0 < t < 1$ the result follows because a subarc of a geodesic is again a geodesic.
\end{proof}

\section{Metric Properties of the Operator $(u,v) \to P(u,v)$}

In this short section we explore the geometry of the operator $(u,v) \to P(u,v)$ restricted to the space $\mathcal E^2(X,\o)$. First we observe that the triplet $(u,v,P(u,v))$ always forms a right triangle. This will help in proving that $P(\cdot,\cdot)$ contracts distances with respect to $\tilde d $ in both components.

\begin{proposition}[Pythagorean formula] \label{pythagorean} Given $u_0,u_1 \in \mathcal E^2(X,\o)$, we have $P(u_0,u_1)\in \mathcal E^2(X,\o)$ and
$$\tilde d(u_0,u_1)^2 = \tilde d(u_0,P(u_0,u_1))^2 + \tilde d(P(u_0,u_1),u_1)^2.$$
\end{proposition}
\begin{proof} By Proposition \ref{Mdist_est_gen}, it is enough to prove the above formula for $u_0,u_1 \in \mathcal H$. According to Theorem \ref{DR_reg} we have $P(u_0,u_1) \in \mathcal H_\Delta$. Suppose $(0,1) \ni t\to u_t \in \mathcal H_\Delta$ is the geodesic connecting $u_0,u_1$. By (\ref{distgeod}):
$$d(u_0,u_1)^2 = \int_X \dot u_0^2 (\o + i\partial \bar \partial u_0)^n.$$
To complete the argument we will prove the following:
\begin{equation}\label{u1dist}
d(u_1, P(u_0,u_1))^2 = \int_{\{\dot u_0 >0\}} \dot u_0^2 (\o + i\partial \bar \partial u_0)^n,
\end{equation}
\begin{equation}\label{u0dist}
d(u_0, P(u_0,u_1))^2 = \int_{\{\dot u_0 < 0\}} \dot u_0^2 (\o + i\partial \bar \partial u_0)^n.
\end{equation}
We prove now (\ref{u1dist}). Using Lemma \ref{sublevel_lemma} we can write:
\begin{flalign*} \int_{\{\dot u_0 > 0\}} \dot u_0^2 (\o + i\partial \bar \partial u_0)^n&= 2\int_0^{\infty} \tau(\o + i\partial \bar \partial u_0)^n(\{{\dot u_0 \geq \tau}\})d\tau\\
&= 2\int_0^{\infty} \tau(\o + i\partial\bar\partial u_0)^n(\{ P(u_0,u_1 - \tau)=u_0\})d\tau.
\end{flalign*}
Suppose $(0,1) \ni t\to \tilde u_t \in \mathcal H_\Delta$ is the weak geodesic connecting $P(u_0,u_1),u_1$. As $(t,x) \to \tilde u_t(x)$ is increasing in the $t-$variable, we have $\dot{\tilde {u}}_0 \geq 0$. By Theorem \ref{distgeod_general}, Lemma \ref{sublevel_lemma} and Proposition \ref{MA_form} we can write
\begin{flalign*} \tilde d (P(u_0,u_1),u_1)^2 &=\int_X \dot{\tilde {u}}_0^2 (\o + i\partial \bar \partial P(u_0,u_1))^n=\int_{\{\dot{\tilde {u}}_0 > 0\}} \dot{\tilde {u}}_0^2 (\o + i\partial \bar \partial P(u_0,u_1))^n\\
&= 2\int_0^{\infty} \tau(\o + i\partial\bar\partial P(u_0,u_1))^n(\{ \dot{\tilde {u}}_0 \geq \tau\})d\tau\\
&= 2\int_0^{\infty} \tau(\o + i\partial\bar\partial P(u_0,u_1))^n(\{ P(P(u_0,u_1),u_1 - \tau)=P(u_0,u_1)\})d\tau\\
%\end{flalign*}
%\begin{flalign*}
%\hspace{.4in}
&= 2\int_0^{\infty} \tau(\o + i\partial\bar\partial P(u_0,u_1))^n(\{ P(u_0,u_1-\tau)=P(u_0,u_1)\})d\tau\\
&= 2\int_0^{\infty} \tau(\o + i\partial\bar\partial u_0)^n(\{ P(u_0,u_1-\tau)=P(u_0,u_1)=u_0\})d\tau\\
&= 2\int_0^{\infty} \tau(\o + i\partial\bar\partial u_0)^n(\{ P(u_0,u_1-\tau)=u_0\})d\tau,
\end{flalign*}
where in the third line we have used Lemma \ref{sublevel_lemma}, in the fifth line we have used Proposition \ref{MA_form} and the fact that $\{ P(u_0,u_1)=u_1\}\cap\{ P(u_0,u_1 -\tau)=P(u_0,u_1)\}$ is empty for $\tau>0$.

From our calculations (\ref{u1dist}) follows. One can conclude (\ref{u0dist}) from (\ref{u1dist}) after reversing the roles of $u_0,u_1$ and then using (\ref{lengthequal1}).
\end{proof}
\begin{proposition} \label{contractivity} Given $u,v,w \in \mathcal E^2(X,\o)$ we have
$$\tilde d (P(u,v),P(u,w)) \leq \tilde d(v,w).$$
\end{proposition}
\begin{proof}
First we assume that $v \leq w$. As before, we can also assume $u,v,w \in \mathcal H$ and $P(u,v),P(u,w) \in \mathcal H_\Delta$. Let $(0,1) \ni t \to \phi_t,\psi_t \in \mathcal H_\Delta$ be geodesic segments, $\phi_t$ connecting $v,w$ and $\psi_t$ connecting $P(u,v),P(u,w)$. By Theorem \ref{distgeod_general} we have to argue that
\begin{equation}\label{dist_est}
\int_X \dot{\psi^2_0}(\o + i \partial\bar\partial P(u,v))^n\leq\int_X \dot{\phi^2_0}(\o + i \partial\bar\partial v)^n.
\end{equation}
Proposition \ref{MA_form} implies that
$$\int_X \dot{\psi^2_0}(\o + i \partial\bar\partial P(u,v))^n\leq \int_{\{P(u,v)=u\}} \dot{\psi^2_0}(\o + i \partial\bar\partial u)^n + \int_{\{P(u,v)=v\}} \dot{\psi^2_0}(\o + i \partial\bar\partial v)^n$$
We argue that the first term in this sum is zero. As $P(u,v) \leq P(u,w)$, it is clear that $t \to \psi_t$ is increasing in $t$. By the maximum principle, it is also clear that $\psi_t \leq u, \ t \in [0,1]$. Hence, if $x \in \{ P(u,v)=u\}$ then $\psi_t(x)=u(x), t \in [0,1]$, implying $\dot \psi_0 \big|_{\{ P(u,v)=u\}}\equiv 0$.

At the same time, using the maximum principle again, it follows that $\psi_t \leq \phi_t, \ t \in [0,1]$. This implies that $0 \leq \dot \psi_0 \big|_{\{ P(u,v)=v\}}\leq\dot \phi_0 \big|_{\{ P(u,v)=v\}},$ which in turn implies (\ref{dist_est}).

The general case follows now from an application of the Pythagorean formula (Proposition \ref{pythagorean}) and what we just proved:
\begin{flalign*}
\tilde d(P(u,v),P(u,w))^2&= \tilde d(P(u,v),P(u,v,w))^2 + d(P(u,w),P(u,v,w))^2\\
&= \tilde d(P(u,v),P(u,P(v,w)))^2 + d(P(u,w),P(u,P(v,w)))^2\\
&\leq \tilde d(v,P(v,w))^2 + d(w,P(v,w))^2\\
&= \tilde d(v,w)^2.
\end{flalign*}
\end{proof}

\section{Completeness of $(\mathcal E^2(X,\o), \tilde d)$}

We recall that the Aubin-Mabuchi energy is a functional $AM : \mathcal H_0 \to \Bbb R$ defined by the formula:
$$AM(v)=\frac{1}{n+1}\sum_{j=0}^n\int_{X} v\o^j\wedge (\o + i\partial\bar\partial v)^{n-j}.$$
As an easy computation shows, for $u,v \in \mathcal H_0$ we have
\begin{equation}\label{AM_diff}AM(u)-AM(v)=\frac{1}{n+1} \sum_{j=0}^n\int_{X} (u-v)(\o +i\partial\bar\partial u)^j\wedge (\o + i\partial\bar\partial v)^{n-j}.
\end{equation}
As the growth of $AM(\cdot)$ is the same as the growth of $E_{\chi^1}(\cdot)$ (\cite[Proposition 2.8]{begz}), one can extend $AM(\cdot)$ to $\mathcal E^1(X,\o)$. We observe now that the Aubin-Mabuchi energy is Lipschitz continuous with respect to the Mabuchi metric:

\begin{lemma} \label{AMcont} Given $u_0,u_1 \in \mathcal E^2(X,\o)$, we have $|AM(u_0) - AM(u_1)| \leq \sqrt{\textup{Vol}(X)}\tilde d(u_0,u_1)$.
\end{lemma}
\begin{proof} By density we can suppose that $u_0,u_1 \in \mathcal H.$ Let $(0,1) \ni t \to u_t \in \mathcal H_\Delta$ be the geodesic connecting $u_0,u_1$. By (\ref{distgeod}), (\ref{AM_diff}) and the Cauchy-Schwarz inequality we have:
\begin{flalign*}
AM(u_1) - AM(u_0)&= \int_0^1 \frac{d AM(u_t)}{dt}dt=\int_0^1 \int_X \dot {u_t}(\o + i\partial\bar\partial u_t)^n dt\\
&\leq \int_0^1 \sqrt{\textup{Vol}(X)}\sqrt{\int_X \dot {u_t^2}(\o + i\partial\bar\partial u_t)^n} dt=\sqrt{\textup{Vol}(X)}d(u_0,u_1).
\end{flalign*}
\end{proof}

We are ready to prove completeness of $(\mathcal E^2(X,\o),\tilde d)$. Roughly, the idea of the proof is to replace an arbitrary Cauchy sequence with an equivalent monotone Cauchy sequence which is much easier to deal with in light of Proposition \ref{Mdist_est_gen} and Lemma \ref{mononton_seq}.

\begin{theorem} \label{E2complete} $(\mathcal E^2(X,\o),\tilde d)$ is complete.
\end{theorem}
\begin{proof} In Corollary \ref{supspaces} we have seen that $\mathcal H$ is a dense subset of $\mathcal E^2(X,\o)$. Suppose $\{u_k\}_{k \in \Bbb N} \subset \mathcal H$ is a $\tilde d-$Cauchy sequence with respect to the Mabuchi metric. We will prove that there exists $v \in \mathcal E^2(X,\o)$ such that $\tilde d (v,u_k) \to 0.$ After passing to a subsequence we can assume that
$$\tilde d(u_l,u_{l+1}) \leq 1/2^l, \ l \in \Bbb N.$$
We introduce $v^k_l = P(u_k,u_{k+1},\ldots,u_{k+l}) \in \mathcal H_0, \ l,k \in \Bbb N$. We argue first that each decreasing sequence $\{ v^k_l\}_{l \in \Bbb N}$ is $\tilde d-$Cauchy. Given our assumptions, this fill follow if we show that $\tilde d(v^k_{l+1},v^k_l) \leq \tilde d (u_{l+1},u_{l}).$
We observe that $v^k_{l+1}=P(v^k_l,u_{k+ l+1})$ and $v^k_l=P(v^k_l,u_{k+l})$. Using this and Proposition \ref{contractivity} we can write:
$$ \tilde d(v^k_{l+1},v^k_l) = \tilde d(P(v^k_l,u_{k+l+1}),P(v^k_l,u_{k+l})) \leq \tilde d(u_{k+l+1}, u_{k+l})\leq \frac{1}{2^{k+l}}.$$

As we show below in Lemma \ref{mononton_seq}, it follows now that each sequence $\{ v^k_l\}_{l \in \Bbb N}$  is $\tilde d-$convergening to some $v^k \in \mathcal E^2(X,\o)$. Using the same trick as above, one can prove:
$$\tilde d(v^k,v^{k+1}) =\lim_{l \to \infty}\tilde d(v^k_{l+1},v^{k+1}_l)= \lim_{l \to \infty}\tilde d(P(u_k,v^{k+1}_{l}),P(u_{k+1},v^{k+1}_l))\leq \tilde d (u_k,u_{k+1}) \leq \frac{1}{2^k},$$
\begin{flalign*}
\tilde d(v^k,u_k) &=\lim_{l \to \infty}\tilde d(v^k_l,u_k)=\lim_{l \to \infty}\tilde d((P(u_k,v^{k+1}_{l-1}),P(u_k,u_k))\\
&\leq\lim_{l \to \infty}\tilde d(v^{k+1}_{l-1},u_k)=\lim_{l \to \infty}\tilde d(P(u_{k+1},v^{k+2}_{l-2}),u_k)\\
&\leq \lim_{l \to \infty}\tilde d(P(u_{k+1},v^{k+2}_{l-2}),u_{k+1}) + \tilde d(u_{k+1},u_k)\\
&\leq \lim_{l \to \infty} \sum_{j=k}^{l+k}\tilde d (u_j,u_{j+1}) \leq \frac{1}{2^{k-1}}.
\end{flalign*}Hence, $\{v^k\}_{k \in \Bbb N}$ is an increasing $\tilde d-$Cauchy sequence that is equivalent to $\{u_k\}_{k \in \Bbb N}$. We want to show that $\{v^k\}_{k \in \Bbb N}$ increases pointwise a.e. to some $v \in \mathcal E^2(X,\o)$.

Using Lemma \ref{AMcont} and (\ref{AM_diff}) we have:
$$0 \leq \frac{1}{n+1}\int_X (v^k - v^1)(\o + i\partial \bar\partial v^1 )^n \leq AM(v^k) - AM(v^1) \leq \tilde d(v^k,v^1).$$

Hence, the limit $\tilde v=\lim_{k \to \infty} v^k-v^1\geq 0$ is finite on a set of capacity non-zero. Indeed, by the monotone convergence theorem $(\o + i\partial \bar\partial v^1 )^n(\{ \tilde v = \infty\})=0$, implying $(\o + i\partial \bar\partial v^1 )^n(\{ \tilde v < \infty\})= \text{Vol}(X) >0$. As $v^1 \in \mathcal E^2(X,\o) \subset \mathcal E(X,\o)$, by \cite[Theorem A]{gz} it follows that $\text{Cap}(\{ \tilde v < \infty\}) >0$.

Let $A_l = \{ \tilde v < l, v^1 > -l\}, \ l >0$. As $\text{Cap}(\{  v^1 = -\infty\}) =0$ and $\{ \tilde v < \infty, v^1 > -\infty\} = \cup_{l >0} A_l$, we conclude that $\text{Cap}(A_{l_0}) >0$ for some $l_0 >0$ and
$$-l_0 < \sup_{A_{l_0}} v^k < l_0 + \sup_{A_{l_0}}v^1, k \in \Bbb N.$$
By \cite[Corollary 4.3]{gz2}, this implies that $\{v^k\}_{k \in \Bbb N}$ forms an $L^1-$relatively compact family in $\text{PSH}(X,\o)$. It follows that $v_k$ increases a.e. to some $ v \in \text{PSH}(X,\o)$, in particular $v \in \mathcal E^2(X,\o)$. An application of Proposition \ref{Mdist_est_gen} now yields $\tilde d(v^k,v) \to 0$, which in turn implies $\tilde d(u_k,v) \to 0$.
\end{proof}

As promised, we need to argue that decreasing Cauchy sequences in $(\mathcal E^2(X,\o),\tilde d)$ have their limit in $\mathcal E^2(X,\o)$. Before we do this, we state a lemma of independent interest:
\begin{lemma} Suppose $u \in \mathcal H$ and $u \leq 0$. Then $u/2 \in \mathcal H$ and
$$d(0,u/2) \leq  d(0,u)/2.$$
\end{lemma}
\begin{proof} Suppose $(0,1) \ni t \to \phi_t,\psi_t \in \mathcal H_\Delta$ is the geodesic connecting $0,u$ and $0,u/2$ respectively. Then $t \to \phi_t/2$ is a subgeodesic connecting $0, u/2$.
By (\ref{distgeod}) we have $\tilde d(0,u/2)^2 = \int_X \dot \psi_0^2 \o^n$ and $\tilde d(0,u)^2 = \int_X \dot \phi_0^2 \o^n$. As $\phi_t/2 \leq \psi_t , \ t \in [0,1]$ and both curves are decreasing the desired estimate follows.
\end{proof}

\begin{lemma}\label{mononton_seq} Suppose $\{u_k\}_{k \in \Bbb N} \subset \mathcal E^2(X,\o)$ is a pointwise decreasing $\tilde d-$bounded sequence. Then $u = \lim_{k \to \infty} u_k \in \mathcal E^2(X,\o)$ and additionally $\tilde d(u,u_k) \to 0$.
\end{lemma}
\begin{proof} We can suppose that $u_k < 0$. First we assume that $\{u_k\}_{k \in \Bbb N} \subset \mathcal H$. By Lemma \ref{AMcont} it follows that $u = \lim_{k \to \infty} u_k \in \mathcal E^1(X,\o)$. Suppose $v \in \mathcal E^2(X,\o)$ and $v < 0$. We will prove that
\begin{equation} \label{E2est}
\int_X v^2 (\o + i\partial \bar \partial u)^n < + \infty.
\end{equation}
Suppose $v_l \searrow v$ is a decreasing sequence of negative smooth K\"ahler potentials. As $0 > v_l \geq 2(P(v_l,u_k) - u_k/2)$ and $(\o + \partial \bar \partial u_k)^n \leq 2^n (\o + \partial \bar \partial u_k/2)^n$ we can start writing:
\begin{flalign*}
\int_X v^2_l (\o + i\partial \bar \partial u_k)^n & \leq  2^{n+2}\int_X (P(v_l,u_k) - u_k/2)^2 (\o + i\partial \bar \partial u_k/2)^n\\
&\leq 2^{n+2}\tilde d(P(v_l,u_k),u_k/2)^2\\
&\leq 2^{n + 3}(\tilde d(P(v_l,u_k),u_k)^2 + \tilde d(u_k,u_k/2)^2)\\
&\leq 2^{n + 3}(\tilde d(v_l,u_k)^2 + \tilde d(u_k,u_k/2)^2)\\
&\leq 2^{n + 4}(\tilde d(v_l,0)^2 + 2\tilde d(u_k,0)^2 + \tilde d(u_k/2,0)^2),
\end{flalign*}
where in the second line we have used Lemma \ref{Mdist_est} and in the forth line we have used the Pythagorean formula (Proposition \ref{pythagorean}). As both $\{v_l\}_{l \in \Bbb N}$ and $\{u_k\}_{k \in \Bbb N}$ are $\tilde d-$bounded sequences, using the previous lemma we conclude that the right hand side above is bounded. Letting $k \to \infty$ and then $l \to \infty$ we obtain (\ref{E2est}). By \cite[Theorem C]{gz} and Theorem \ref{dinew} it follows that $u \in \mathcal E^2(X,\o)$. The last statement follows from Proposition \ref{Mdist_est_gen}.

The case when $\{u_k\}_{k \in \Bbb N} \subset \mathcal E^2(X,\o)$ can be reduced to the above situation using approximation via Proposition \ref{Mdist_est_gen}.
\end{proof}

\vspace{0.1 in}
\textsc{Department of Mathematics, Purdue University}\\
\emph{E-mail address: }\texttt{\textbf{tdarvas@math.purdue.edu}}
\end{document}